\documentclass{amsart}

\newcommand{\alphlist}{\begin{list}{(\alph{enumi})}{\usecounter{enumi}\setlength{\parsep}{2pt}
      \setlength{\itemsep}{1pt} \setlength{\topsep}{5pt}
      \setlength{\partopsep}{3pt}}}
\newcommand{\arablist}{\begin{list}{(\arabic{enumi})}{\usecounter{enumi}\setlength{\parsep}{2pt}
          \setlength{\itemsep}{1pt} \setlength{\topsep}{5pt}
          \setlength{\partopsep}{3pt}}}
\newcommand{\romanlist}{\begin{list}{(\roman{enumi})}{\usecounter{enumi}\setlength{\parsep}{2pt}
              \setlength{\itemsep}{1pt} \setlength{\topsep}{5pt}
              \setlength{\partopsep}{3pt}}}
\newcommand{\Romanlist}{\begin{list}{(\Roman{enumi})}{\usecounter{enumi}\setlength{\parsep}{2pt}
              \setlength{\itemsep}{1pt} \setlength{\topsep}{5pt}
              \setlength{\partopsep}{3pt}}}
\newcommand{\bulletlist}{\begin{list}{$\bullet$}{\setlength{\parsep}{2pt}
                \setlength{\itemsep}{1pt} \setlength{\topsep}{5pt}
                \setlength{\partopsep}{3pt}\setlength{\leftmargin}{15pt}}} 
\newcommand{\Alphlist}{\begin{list}{(\Alph{enumi})}{\usecounter{enumi}\setlength{\parsep}{2pt}
      \setlength{\itemsep}{1pt} \setlength{\topsep}{5pt}
      \setlength{\partopsep}{3pt}}}
 \newcommand{\listend}{\end{list}}

\newcommand{\T}{\ensuremath{\mathbb{T}}}

\newcommand{\N}{\ensuremath{\mathbb{N}}} 
\newcommand{\R}{\ensuremath{\mathbb{R}}}
\newcommand{\Z}{\ensuremath{\mathbb{Z}}}
\newcommand{\Q}{\ensuremath{\mathbb{Q}}}

\usepackage[utf8]{inputenc}
\usepackage{amssymb,amsthm, amsmath,amscd}
\usepackage[english]{babel}
\usepackage{indentfirst}
\usepackage{scalefnt}
\usepackage{multicol}
\usepackage{tikz}
\usepackage{enumerate}
\usepackage{xcolor}

\newtheorem*{teoa}{Theorem A}
\newtheorem{teo}{Theorem}
\newtheorem{lema}[teo]{Lemma}
\newtheorem{corolario}[teo]{Corollary}
\newtheorem*{corb}{Corollary B}
\newtheorem{proposicao}[teo]{Proposition}

\theoremstyle{definition}
\newtheorem{definicao}[teo]{Definition}

\theoremstyle{remark}
\newtheorem{obs}[teo]{Remark}

\newcommand{\F}{\mathcal{F}}
\newcommand{\g}{\tilde \gamma_0}
\newcommand{\bx}{\beta_{\tilde{x}}}
\newcommand*\dif{\mathop{}\!\mathrm{d}}
\renewcommand{\T}{\mathbb{T}^2}

\begin{document}

\title[No sublinear diffusion for a class of torus homeomorphisms]{Inexistence of sublinear diffusion for a class of torus homeomorphisms}
\author{Guilherme Silva Salomão}\thanks{G. S. S. was supported by Fapesp,  F. T. was partially supported by the Alexander von Humboldt foundation and by Fapesp and CNPq.}
\author{Fabio Armando Tal}
\address{Instituto de Matemática e Estatística, Rua do Mat\~ao 1010, Cidade Universitária, São Paulo, SP, Brazil, 05508-090}
\email{salomao.guilherme@gmail.com,\, fabiotal@ime.usp.br}

\begin{abstract}
 We prove that, if $f$ is a homeomorphism of the two torus isotopic to the identity whose rotation set is a non-degenerate segment and $f$ has a periodic point, then it has uniformly bounded deviations in the direction perpendicular to the segment.
\end{abstract}

\maketitle
\section{Introduction}
The study of surface dynamics from a topological viewpoint has been gathering increasing attention in the last decade, in large part because of the developments of new tools and techniques that have been proven effective in tackling previously hopeless problems. A great deal of these developments have been tied to recent improvements in both Brouwer theory and Rotation theory. In particular, the search for a greater understanding of the dynamics of torus homeomorphisms in the isotopy class of the identity has been one of the motivating forces behind the developments, due to its connection to relevant physical dynamics like Hamiltonian homeomorphisms in general or specific models as the  Kicked-Harper model and the Zaslavsky-Web maps. 

Rotation theory for homeomorphisms of the $2$-torus $\T=\R^2/\Z^2$ appeared in the early 90s as an extension of the ideas of the Poincaré rotation number for homeomorphisms of the circle. Given $f:\T\to\T$ a torus homeomorphism isotopic to the identity, and $\widetilde{f}:\R^2\to\R^2$ a lift of $f$ to the universal covering of $\T$ one can define, following \cite{MisiurewiczZiemian1989RotationSets}, the {\it{rotation set}} $\rho(\widetilde{f})$ of $\widetilde{f}$ as:
$$\rho(\widetilde{f}):=\{v \mid \exists n_k\to \infty, \widetilde{z}_k\in \R^2, \lim_{k\to\infty}\frac{\widetilde{f}^{n_k}(\widetilde{z}_k)-\widetilde{z}_k}{n_k}=v\},$$
which is always compact and convex.
This notion, which is invariant by change of coordinates in the isotopy class of the identity,  has been shown to be of great utility in describing the dynamics of these maps.  For instance, it is in some cases possible to deduce, just by the analysis of the rotation sets, that the dynamics has periodic points of arbitrarily large period \cite{franks:1989}, that is has positive entropy \cite{llibre/mackay:1991}, or even that it has a well defined chaotic region \cite{KoropeckiTal2012StrictlyToral}.

One relevant feature of rotation sets is that they, as per the definition, describe only linear rates of displacement in the lift. But a question that has appeared in several contexts is to determine if it is possible that sublinear displacements can exists that are not captured by it. For, while it follows directly from the definitions that there exists $M_0, N_0>0$ such that for all $\widetilde{z}$ in $\R^2$ and all $n>N_0$, $d\left( (\widetilde{f}^n(\widetilde{z})-\widetilde{z})/n, \rho(\widetilde{f})\right) <M_0$, one is left to wonder if there it is also possible to obtain a better estimate. Specifically, does it also hold that there exists some $M_1>0$ such that for all $\widetilde{z}$ in $\R^2$ and all $n$, $d\left( \widetilde{f}^n(\widetilde{z})-\widetilde{z}), n \rho(\widetilde{f}) \right) <M_1$? If the latter occur, we say that $f$ has uniformly bounded deviation from its rotation set. A similar concept is that of bounded deviations in a direction $w$. Let $w\in\R^2_*$, define the projection in the $w$ direction as $P_w:\R^2\to\R^2, P_w(x)=\langle x, w/\Vert w \Vert\rangle$. One says that $f$ has {\it{ bounded $w$-deviations}} if there exists $M_1>0$ such that  for all $\widetilde{z}$ in $\R^2$ and all $n$, $d\left( P_w(\widetilde{f}^n(\widetilde{z})-\widetilde{z}), P_w(n \rho(\widetilde{f}) \right) <M_1$.

Analysis of bounded deviations for homeomorphisms of $\T$ is a topic increasingly present in the literature. This is due both to its intrinsic interest, but also to its use a fundamental tool in solving some traditional problems in the field. For instance, it has appeared in the proof of Boyland's conjecture in the torus and in the closed annulus \cite{AddasZanata2015BoundedMeanMotionDiffeos, LeCalvezTal2015ForcingTheory, ConejerosTal2}, in determining the existence of Aubry-Mather sets for homeomorphisms of the closed annulus \cite{ConejerosTal2}, in the study of the existence of irrational rotation factors for the dynamics \cite{Jaeger2009Linearisation, JaegerTal2016IrrationalRotationFactors, JaegerPasseggi2015SemiconjugateToIrrational, kocsard2017rotational} and in the attempts to solve the remaining case of the Franks-Misiurewicz Conjecture \cite{KoropeckiPasseggiSambarino2016FMC, passeggi2018deviations, kocsard2016dynamics}. While it is known that bounded deviations don't necessarily hold in all situations, in particular when the rotation set is a singleton  (see \cite{KoropeckiTal2012Irrotational, kocsard/koropecki:2007}), there are several cases where it can be established as, for instance, if $\rho(\widetilde{f})$ has nonempty interior (\cite{AddasZanata2015BoundedMeanMotionDiffeos, LeCalvezTal2015ForcingTheory}). Furthermore, it is also known (see \cite{Davalos2013SublinearDiffusion, GuelmanKoropeckiTal2012Annularity} that if the rotation set is a line segment with two different points with rational coordinates, then $f$ has bounded deviations in the direction that is perpendicular to the the segment.

In all these studies, one case which remained unsolved in either direction is to determine, wherever the rotation set was a non-degenerate line segment with at most one rational point, if bounded deviations for a given direction still needed to hold. Part of the problem here is the relatively lack of examples of these situations. For instance, it was not known until the recent work by Avila if there existed a homeomorphism whose rotation set was a non-degenerate line segment without rational points. On the other hand, while there existence of examples of rotation sets which are non-degenerate line intervals with a single rational point is a folklorical result, these examples always had the rational point as an extremity of the line segment. Indeed, in \cite{ LeCalvezTal2015ForcingTheory} it was shown that this must be the case, that is, that there is no homeomorphism of $\T$ such that its rotation set is a non-degenerate line segment with a single bi-rational point in its relative interior.

In this paper we deal exactly with this latter situation. Let us denote by $(x)_1$ and $(x)_2$ the canonical first and second coordinates, respectively, of a point $x\in\R^2$, and given $\rho_0=((\rho_0)_1,(\rho_0)_2)$, let us denote $\rho_0^\perp=(-(\rho_0)_2,(\rho_0)_1)$ and, if $(\rho_0)_1\not=0$, denote $\tan(\rho_0)=(\rho_0)_2/(\rho_0)_1$. Our main result is that bounded deviation in the perpendicular direction must exists:

\begin{teoa}\label{teoremaA}
Let $f:\T\to\T$ be a torus homeomorphism isotopic to the identity, $\tilde{f}:\R^2\to\R^2$ a lift of $f$ and $\pi:\R^2\to\T$ the covering map.
Suppose that $\rho(\tilde f)=\{t\rho_0\mid 0\leq t\leq 1\}$, where $\tan(\rho_0)\notin\mathbb{Q}$. Then  there is $M>0$ such that $$|\langle \tilde{f}^n(\tilde{z})-\tilde{z},\rho_0^\perp\rangle|<M,$$ for every $\tilde{z}\in\R^2$ e $n\in\Z$.
\end{teoa}

An immediate corollary, using the results from D\'avalos (\cite{Davalos2013SublinearDiffusion}) and Le Calvez and the second author (\cite{LeCalvezTal2015ForcingTheory})) is that

\begin{corb}
Let $f:\T\to\T$ be a torus homeomorphism isotopic to the identity, $\tilde{f}:\R^2\to\R^2$ a lift of $f$ and $\pi:\R^2\to\T$ the covering map.
Suppose that $\rho(\tilde f)$ is a non-degenerate line segment and that $f$ has at least one periodic point. Then $f$ has bounded deviations in the direction perpendicular to $\rho(\tilde f)$. 
\end{corb}

Theorem A  should have plenty of applications and has already been used in \cite{xioachuansalvador2019}.
The main new technical development that allowed for this work was the introduction of the new forcing techniques for surface homeomorphisms in \cite{ LeCalvezTal2015ForcingTheory}, although in no way its just a straightforward application. Also, the theory is relatively recent, and this work has as a subproduct some new lemmas that may be useful in its application. The paper is organized as follow. In Section~\ref{sec:preliminaries} we describe the main tools necessary for our work and in Section~\ref{sec:mainresult} we prove the main theorem and its corollary.


\section{Preliminaries}\label{sec:preliminaries}

\subsection{Rotation theory for homeomorphisms of $\T$}
\label{caprot}

Let $f:\T\to\T$ be an homeomorphisms isotopic to the identity, where $\T=\R^2/\Z^2$, let $\tilde{f}:\R^2\to\R^2$ be a lift of $f$ to the universal covering of $\T$ and let $\pi:\R^2\to\T$ be the covering map. We already defined the  rotation set of $\tilde f$ at the introduction, also known as the 
 \emph{ Misiurewicz-Ziemian rotation set}. We say that a point $x\in\T$ has a \emph{rotation vector} $v$ if $\rho(\tilde f,x)=\lim_{n\to\infty}\frac{\tilde{f}^n(\tilde{x})-\tilde{x}}{n}=v,$ where $\tilde{x}$ is any point in $\pi^{-1}(x)$.
 
Let $\phi:\T\to\R^2$ be the displacement function $\phi(x)=\tilde{f}(\tilde x)-\tilde x$, where $\tilde x$ is some point in $\pi^{-1}(x)$. As $$\frac{1}{n}\sum_{k=0}^{n-1}\phi(f^k(x))=\frac{1}{n}\sum_{k=0}^{n-1}(\tilde{f}^{k+1}(\tilde x)-\tilde{f}^k(\tilde x))=\frac{\tilde{f}^n(\tilde x)-\tilde x}{n},$$ we have, by Birkhoff's Ergodic Theorem, that if $\mu$ is an ergodic borelian probability measure invariant by $f$, then $$\lim_{n\to\infty}\frac{1}{n}\sum_{k=0}^{n-1}\phi(f^k(x))=\int_{\T}\phi \dif\mu \quad\textrm{ for $\mu$ almost all $x\in\T$},$$ therefore allmost all points for $\mu$ have a rotation vector and it is equal to $\int_{\T}\phi \dif\mu$, which is called the \emph{rotation vector of the measure $\mu$}. It is well known also that (see \cite{MisiurewiczZiemian1989RotationSets})
$$\rho(\tilde f)=\left\{\int_{\T}\phi\dif\mu\mid\mu \textrm{ is a borel probability measure invariant by $f$}\right\},$$
which bridges the concept of rotation for points and the displacement for invariant measures. As a consequence, one obtains that rotation sets are always compact and convex subsets of the plane, and that if $v$ is an extremal point of the rotation set, then there exists an ergodic $f$-invariant measure $\mu$ whose rotation vector is $v$. 

Of particular importance for this work, we have that, whenever the rotation set of $\tilde f$ is a line segment with irrational slope that contains the point $(0,0)$, one must have that $(0,0)$ is an extremal point of $\rho(\tilde f)$ (by \cite{LeCalvezTal2015ForcingTheory}) and a result from Franks (see \cite{franks:1988a}) implies that $\tilde f$ must have a fixed point. Furthermore, since in this case $(0,0)$ is the only point in $\rho(\tilde f)\cap \Q^2$, then every periodic point of $f$ must be lifted to a periodic point of $\tilde f$.

\subsection{Essential dynamics}
We need the concept of essential points for the dynamics, developed in \cite{KoropeckiTal2012StrictlyToral, koropecki2018fully}. We refer to the papers for a complete exposition, only citing the required results.

An open subset $U\subset\T$ is \textit{inessential} if every closed curve contained in $U$ is null-homotopic in $\T$, otherwise $U$ is called \textit{essential}. A general subset $E\subset\T$ is inessential if it has a inessential open neighborhood, otherwise it is called essential. Finally, $E$ is said to be \emph{fully essential} if $\T\setminus E$ is inessential. 
\begin{definicao}\label{defdiness}
Let $x\in\T$ and $f:\T\to\T$ be a homeomorphism isotopic to the identity. We say that $x$ is \textit{an inessencial point of $f$} if $\cup_{k\in\Z}f^k(U)$ is inessential for some neighborhood $U$ of $x$, otherwise $x$ is called \textit{an essential point for $f$}. Furthermore, we say that $x$ is a \emph{fully essential point of $f$} if $\cup_{k\in\Z}f^k(U)$ is fully essential for any neighborhood $U$ of $x$.
\end{definicao}

The following proposition appeared in \cite{guelman2015rotation}.   $||x||_\infty$ denotes the infinity norm on $\R^2, \, ||x||_\infty=\max\{|(x)_1|,|(x)_2|\}$, where  $(x)_1$ and $(x)_2$ are the first and second canonical coordinates of a point $x\in\R^2$.

\begin{proposicao}
\label{propguelman2015rotation}
Let $O\subset\R^2$ be a connected open set such that $\bigcup_{n\in\Z}f^n(\pi(O))$ is fully essential and such that $\overline{\pi(O)}$ is inessential. Then there exists $M\in\N$ and $K\subset\R^2$ compact such that $[0,1]^2$ is contained in a bounded connected component of $\R^2\setminus K$ and $$K\subset\bigcup_{|i|\leq M,\, ||v||_{\infty}\leq M}\left(\tilde{f}^i(O)+v\right).$$
\end{proposicao}
 We also have that:
\begin{lema}\label{lemaessencial}
Let $f:\T\to\T$ be an homeomorphism isotopic to the identity, $\tilde{f}:\R^2\to\R^2$ a lift and $z_0\in\T$ a recurrent point such that $\rho(\tilde f,z_0)$ does not belong to $\Q^2$. If $\rho(\tilde f)$ is a nondegenerate line segment with irrational slope, then  $z_0$ is a fully essential point for $f$. In particular, for every $\varepsilon>0$, the set $U_{\varepsilon}=\bigcup_{i=0}^{\infty} f^{i}(\pi(B(\varepsilon, \tilde{z}_0)))$ is fully essential.
\end{lema}
\begin{proof}
Assume, for a contradiction, that $z_0$ is an inessential point. Therefore there exists $\varepsilon>0$ such that $U_{\varepsilon}=\bigcup_{i=0}^{\infty} f^{i}(\pi(B(\varepsilon, \tilde{z}_0)))$ is inessential. Note that each connected component of  $U_\varepsilon$ must be contained in a topological open disk. As $U_\varepsilon$ is $f$-invariant, $f$ permutes the connected componets of $U_\varepsilon$. Also, since $z_0$ is recurrent, if $U_\varepsilon^0$ is the connected component of $U_\varepsilon$ containing $z_0$, there must exist  a smallest $N> 0$ such that $f^{N}(U_\varepsilon^0)=U_\varepsilon^0$. Let $w\in\Z^2$ be the integer vector such that $\tilde{f}^{N}(\widetilde U_\varepsilon^0)=\widetilde U_\varepsilon^0+w$, where $\widetilde U_\varepsilon^0$ is the lift of $U_\varepsilon^0$ that contains$\tilde z_0$. As $z_0$ is recurrent, there exists a subsequence $n_k$ such that $f^{n_k}(z_0)\to z_0$. In particular, there exists  $k_0$ such that if $k>k_0$ we have $f^{n_k}(z_0)\in U_\varepsilon^0$. But by the choice of $N$ we have that $f^i(U_\varepsilon^0)\cap U_\varepsilon^0=\emptyset$ if $1\leq i<N$, one deduces that $n_k=p_k N$, for $k>k_0$. Therefore $\tilde{f}^{n_k}(\tilde z_0)=\tilde{f}^{p_kN}(\tilde z_0)\in \tilde{f}^{p_kN}(\widetilde U_\varepsilon^0)=\widetilde U_\varepsilon^0 +p_kw$. So, $\tilde{f}^{p_kN}(\tilde z_0)-p_kw\to\tilde z_0$, which implies that $\rho(\tilde f,z_0)=w/N$, a contradiction.

Assume now, again for a contradiction, that $z_0$ is essential but not fully essential. There exists $\varepsilon>0$ such that $U_{\varepsilon}=\bigcup_{i=0}^{\infty} f^{i}(\pi(B(\varepsilon, \tilde{z}_0)))$ is essential but not fully essential. Note tha, as $z_0$ is recurrent, all connected components of $U_{\varepsilon}$ are periodic, and by Proposition 1.2 and Proposition 1.4 of \cite{KoropeckiTal2012StrictlyToral}, there must exist $g=f^l$ a power of $f$, $\hat g:\R^2\to\R^2$  a lift of $g$, a vector $W\in \Z^2_*$ , and $M>0$ such that. for all $\tilde z\in\R^2$ and all $n\in\Z$, $|\langle \hat{g}^n(\tilde{z})-\tilde{z}, W\rangle|<M$. But this implies that the rotation set of $\hat g$ is contained in a segment of rational slope. Since $\rho(\hat g)= l \left(\rho(\tilde f)+V\right)$ for some $V\in\Z^2$, one deduces that $\rho(\tilde f)$ is also a line segment of rational slope, again a contradiction.
\end{proof}
\label{capfol}

\subsection{Brouwer homeomorphisms}
We recall that a \emph{Brouwer homeomorphism} is a homeomorphism of the plane preserving orientation without fixed points, and that a \emph{line} in the plane is a continuous, injective and proper map from $\R$ to $\R^2$. By Sch\"{o}enflies Theorem, if $\phi:\R\to\R^2$ is a line, then it can be extended to a homeomorphism $\phi^*:\R^2\to\R^2$ preserving orientation, such that $\phi(t)=\phi^*(t,0)$. We define canonically then the \emph{left and right} of $\phi$ as $L(\phi)=\phi^*(\R\times(0,+\infty))$ and $R(\phi)=\phi^*(\R\times(-\infty,0))$ respectively.

The fundamental result on the study of Brouwer Homeomorphisms is that:
\begin{teo}[\cite{brouwer1912beweis}]\label{teobrouwer0}
Given $h:\R^2\to\R^2$ a Brouwer homeomorphism and $x\in\R^2$, there exists a line $\phi:\R\to\R^2$, with $\phi(0)=x$, such that $h([\phi])\subset L(\phi)$ and $h^{-1}([\phi])\subset R(\phi)$, where $[\phi]=\phi(\R)$.
\end{teo}
A line as in the above result is called a \textit{Brouwer line}. A direct consequence of this theorem is that, if $h$ is a Brouwer homeomorphism, then every point in $\R^2$ is contained in a open invariant connected and simply connected set, and the dynamics of $h$ in this set is conjugated to a rigid translation. In particular, $h$ only has wandering points.






\subsection{Maximal isotopies}
Let $M$ be an oriented surface, $f:M\to M$ a homeomorphisms isotopic to he identity, and denote by $\mathcal{I}$ the space of all isotopies between  $f$  and the identity, that is, if $I\in\mathcal{I}$ then $I=(f_t)_{t\in[0,1]}$, where $f_0=\textrm{Id}_M$, $f_1=f$, for every $t\in[0,1],\, f_t$ is a homeomorphism of $M$, and $I$ is a continuous curve on the space of homeomorphisms of $M$, using the topology of the uniform convergence over compact subsets. The trajectory of a point $z\in M$ for $I$ is defined as the path $t\mapsto f_t(z)$, which we denote by $I(z)$. By concatenating paths, we can define, for $n\in\N$, $$I^n(z)=\Pi_{0\leq k<n}I(f^k(z)),\,\,\, I^\N(z)=\Pi_{k\geq 0}I(f^k(z)) \,\textrm{ e }\, I^\Z(z)=\Pi_{k\in\Z}I(f^k(z)).$$ Denote $\textrm{fix}(I)=\bigcap_{t\in[0,1]}\textrm{fix}(f_t)$ to the set of points whose isotopy path is constant, and let $\textrm{dom}(I)$ be its complement, which is called the  \emph{domain} of $I$. A closed subset $F\subset\textrm{fix}(f)$ is said to be unlinked for $f$ if there exists $I\in\mathcal{I}$ such that $F\subset \textrm{fix}(I)$.

We can define a pre-order in $\mathcal{I}$ as follows:
\begin{definicao}
Let $I_1,I_2\in\mathcal{I}$. Say that $I_1\leqslant I_2$ if
\begin{enumerate}[(i)]
\item $\textrm{fix}(I_1)\subset\textrm{fix}(I_2)$;
\item $I_2$ is homotopic to $I_1$ relative to $\textrm{fix}(I_1)$.
\end{enumerate}
\end{definicao}

We say that $I\in\mathcal{I}$ is a \emph{maximal isotopy} if it is maximal for the pre-order defined above. Note that this is equivalent to the property that, for all $z\in\textrm{fix}(f)\setminus\textrm{fix}(I)$, the trajectory of $z$, which is a closed loop, is not  null homotopic in $\textrm{dom}(I)$ (see \cite{beguin2016fixed}).

Now, if $I$ is a maximal isotopy,  denoting $\tilde{I}=(\tilde f_t)_{t\in[0,1]}$ the lift of $I|_{\textrm{dom}(I)}$ to the universal covering space $\widetilde{\textrm{dom}(I)}$ of $\textrm{dom}(I)$, we have that $\tilde f_1=\tilde f$ has no fixed points and $\tilde f_0=\textrm{Id}_{\widetilde{\textrm{dom}(I)}}$, and therefore the restriction of $\tilde f_1$ to each of the connected components of its domain is a Brouwer homeomorphism. We have the following:

\begin{teo}[\cite{beguin2016fixed}]\label{teobeguin2016fixed}
For all $I\in\mathcal{I}$, there exists $I'\in\mathcal{I}$ such that $I\leqslant I'$ and $I'$ is maximal.
\end{teo}

Therefore, given an unlinked subset $F$ of fixed points for $f$, one can always find a maximal isotopy $I'$ such that $F\subset \textrm{fix}(I')$.

\subsection{Paths transverse to oriented foliations}
If $M$ is a oriented surface, we define a  \emph{oriented topological foliation with singularities of $M$} as a topological oriented foliation $\F$ defined on an open subset of $M$. This subset will be called the \emph{domain of $\F$} and denoted $\textrm{dom}(\F)$, while its complement is the set of  \emph{singularities of $\F$}, and is denoted $\textrm{sing}(\F)$.

Let us fix $\F$ an oriented singular foliation of $M$ and, for each $z \in \textrm{dom}(\F)$, denote by $\phi_z$ to the leaf of $\F$ passing through $z$. A {\it{path}} in $M$ is a continuous function $\gamma:J\to M$ where $J$ is a non-degenerate interval of the line. We denote by $[\gamma]$ the image of $\gamma$.

A path $\gamma:J\to \textrm{dom}(\F)$is said to be transverse to $\F$ if, for all $t\in J$, there exists a homeomorphism $c:W\to (0,1)^2$, where $W$ is a neighborhood of $\gamma(t)$, such that $c$ sends the restriction of the $\F$ to the foliation by downward oriented vertical leafs in $(0,1)^2$, and such that $\pi_1\circ c\circ\gamma$ is strictly increasing in a neighborhood of $t$, where $\pi_1$ is the projection onto the first coordinate. Intuitively, a path is transverse to $\F$ if it always locally crosses leafs from right to left.

\begin{definicao}\label{defequivalence}
If $M=\R^2$ and $\F$ does not have singularities, we say that two transverse paths are \textit{equivalent} if they intersect the same leafs. In the general case we say that two transverse paths $\gamma:J\to\textrm{dom}(\F)$ e $\gamma':J'\to\textrm{dom}(\F)$ are \textit{equivalent} if there exist $H:J\times[0,1]\to\textrm{dom}(\F)$ continuous, and $h:J\to J'$ an increasing homeomorphism such that:
\begin{enumerate}[(i)]
\item $H(t,0)=\gamma(t)$, $H(t,1)=\gamma'(h(t))$;
\item $\forall t\in J$ e $\forall s_1,s_2\in[0,1]$, $\phi_{H(t,s_1)}=\phi_{H(t,s_2)}$.
\end{enumerate}
In this case, we will denote $\gamma\sim_\F\gamma'$.
\end{definicao}

The previous definition is equivalent to showing that, if $\widetilde\F$ is the lift of $\F$ to the universal covering of $\textrm{dom}(\F)$, then there exists lifts $\widetilde\gamma$ and $\widetilde\gamma'$ of $\gamma$ and $\gamma'$, respectively, to the universal covering $\textrm{dom}(\F)$ such that $\widetilde{\gamma}\sim_{\widetilde{\F}}\widetilde{\gamma}'$.

We remark that, by a version of the Poincare-Bendixson Theorem, if $\F$ is a non-singular foliation of $\R^2$ and $\gamma:\R\to\R^2$ is a transverse line, then every leaf $\phi$ intersecting $\gamma$ do so at exacty one point, and the leaf $\phi$ crosses $\gamma$ from left to right, that is, if $t'\in\R$ is such that $\phi(t')\in[\gamma]$, then $\phi(t)\in L(\gamma)$, if $t<t'$, and $\phi(t)\in R(\gamma)$, if $t>t'$.

Given three lines  $\gamma_i:\R\to\R^2$, $i\in\{0,1,2\}$, we say that  \emph{$\gamma_0$ separates $\gamma_1$ and $\gamma_2$} if $[\gamma_1]$ and $[\gamma_2]$ lie in different connected components of the complement of $[\gamma_0]$. We say that \emph{$\gamma_2$  is above $\gamma_1$ with respect to $\gamma_0$} if it holds that:
\begin{enumerate}[(i)]
\item The lines are pairwise disjoint;
\item no line separates the other two;
\item if $\lambda_1$, $\lambda_2$ are two disjoint paths joining $z_1=\gamma_0(t_1)$, $z_2=\gamma_0(t_2)$ to $z'_1\in[\gamma_1]$, $z'_2\in[\gamma_2]$, respectively, and not intersecting the lines but at the extremal points, then $t_2>t_1$.
\end{enumerate}

\begin{figure}[h!t]
\centering
\includegraphics[scale = 0.7]{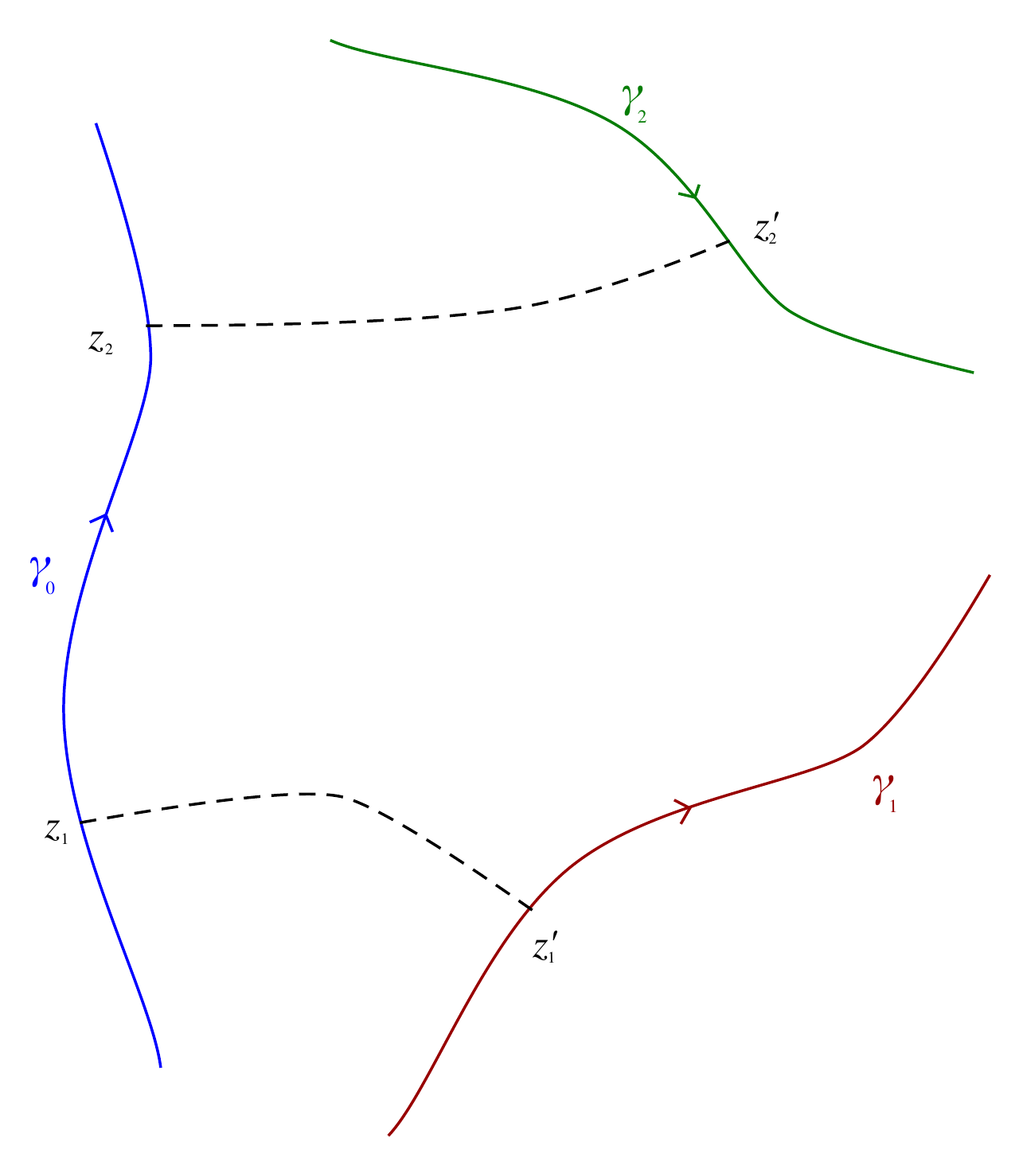}
\caption{$\gamma_2$ is above $\gamma_1$ with respect to $\gamma_0$}
\label{figacima}
\end{figure}

We need the fundamental definition of $\F$-transversal intersection.
\begin{definicao}
Let $M=\R^2$ and let $\F$ be non-singular. Given two transverse paths $\gamma_i:J_i\to\R^2$, $i\in\{1,2\}$, such that  $\phi_{\gamma_1(t_1)}=\phi_{\gamma_2(t_2)}=\phi$ for some $t_1, t_2$ in the interior of $J_1, J_2$ respectively. We say that $\gamma_1$ \emph{intersects $\gamma_2$ $\F$-transversally at $\phi$} if there exists $a_1,b_1\in J_1$ with $a_1<t_1<b_1$ and $a_2,b_2\in J_2$ with $a_2<t_2<b_2$ such that
\begin{enumerate}[(i)]
\item $\phi_{\gamma_2(a_2)}$ is below $\phi_{\gamma_1(a_1)}$ with respect to $\phi$;
\item $\phi_{\gamma_2(b_2)}$ is above $\phi_{\gamma_1(b_1)}$ with respect to $\phi$.
\end{enumerate}

In general, where $M$ is any oriented surface and $\F$ is a singular foliation, we say that two transverse paths $\gamma_i:J_i\to\mathrm{dom}(\F)$, $i\in\{1,2\}$, such that  $\phi_{\gamma_1(t_1)}=\phi_{\gamma_2(t_2)}=\phi$,  \emph{intersects $\gamma_2$ $\F$-transversally at $\phi$}, if there are lifts $\widetilde\gamma_i$ to the universal covering $\widetilde{\mathrm{dom}(\F)}$ that intersect $\widetilde\F$-transversally, where $\widetilde\F$ is the lift of $\F$. 

In both cases, we denote $\gamma_1|_{[a_1,b_1]}\pitchfork_{\F}\gamma_2|_{[a_2,b_2]}$.
\end{definicao}

Whenever the context is clear, we will just say that $\gamma_1$ and $\gamma_2$ intersect transversally. Also, if $\gamma_1(t_1)=\gamma_2(t_2)$ and $\gamma_1$ intersects $\gamma_2$ $\F$-transversally at $\phi_{\gamma_1(t_1)}=\phi_{\gamma_2(t_2)}$, we will just say that $\gamma_1$ and $\gamma_2$ intersect $\F$-transversally at $\gamma_1(t_1)$. In case a path has a transversal intersection with itself, we say that $\gamma$ has a \emph{transversal self-intersection}.

\begin{figure}[h!t]
\centering
\includegraphics[scale = 0.8]{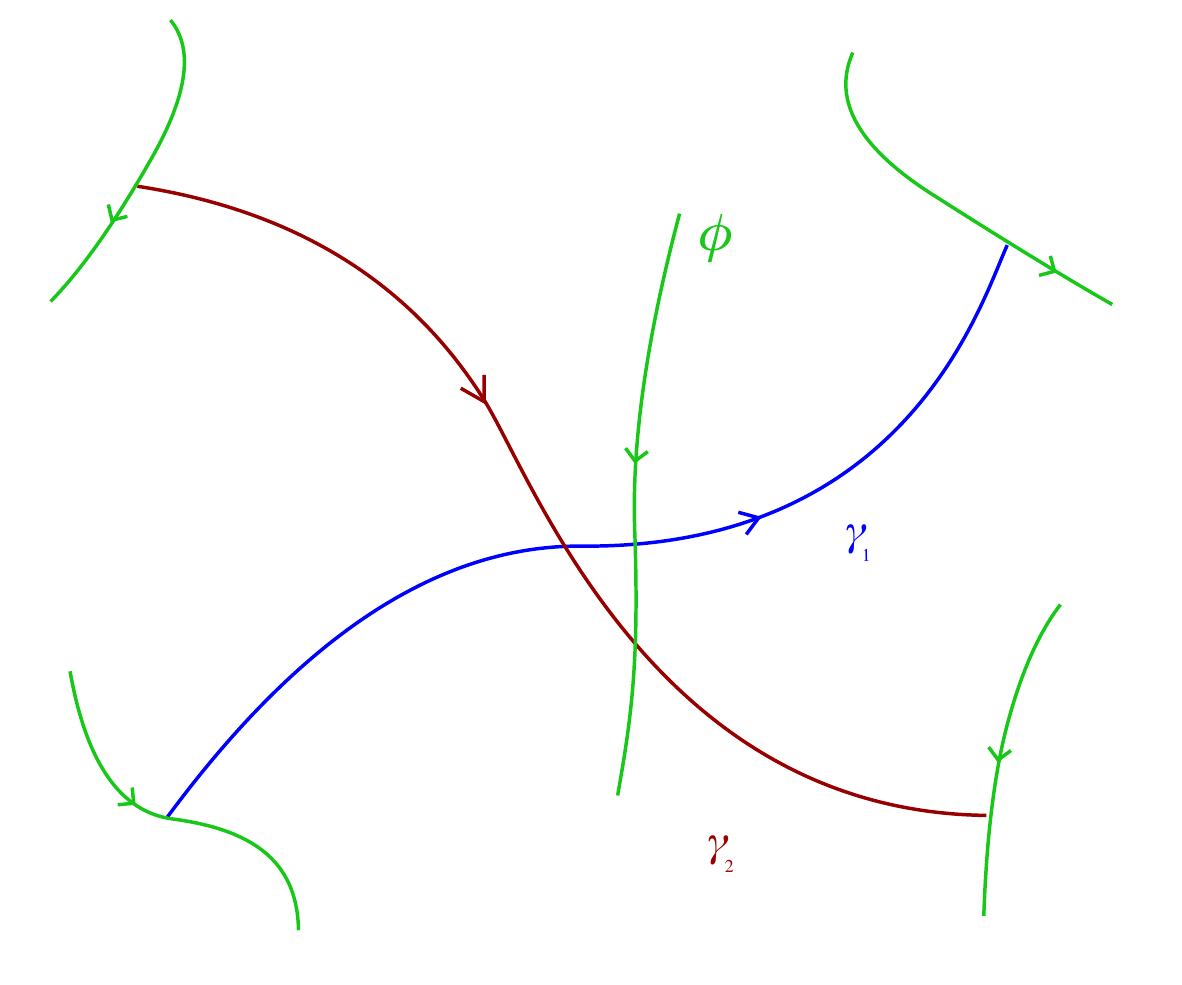}
\caption{$\gamma_1$ and $\gamma_2$ with an $\F$-transversal intersection}
\label{figinter}
\end{figure}

Note that, if  $\gamma_1|_{I_1}\pitchfork_{\F}\gamma_2|_{I_2}$ and $\gamma_2|_{I_2}\sim_{\F}\gamma_3|_{I_3}$, then $\gamma_1|_{I_1}\pitchfork_{\F}\gamma_3|_{I_3}$. Note also that, if $\gamma_1|_{I_1}\pitchfork_{\F}\gamma_2|_{I_2}$ at $\phi$ does not imply that $\gamma_1$ and $\gamma_2$ actually intersect at a point of $\phi$. However, $\gamma_1|_{I_1}$ and $\gamma_2|_{I_2}$ have at least a common point.

\subsection{Brouwer-Le Calvez foliations}

The following is one of the  most useful tools in the study of homeomorphisms of surfaces.

\begin{teo}[\cite{le2005version}]\label{teofolheacao}
Let $h:\R^2\to\R^2$ be a Brouwer homeomorphism and $G$ a discrete group of orientation preserving homeomorphisms of the plane that acts freely and properly. If $h$ commutes with the  elements of $G$, then there exists a foliation $\F$ of $\R^2$ by Brouwer lines for $h$. Furthermore, $\F$ is $G$-invariant.
\end{teo}

A foliation as in the previous theorem will be called a \emph{Brouwer-Le Calvez foliation}.

Note that, by combining the previous result with Theorem \ref{teobeguin2016fixed}, one can obtain that, given a homeomorphism $f$ of $M$ isotopic to identity and an unlinked subset $F$ of $\textrm{fix}(f)$, one can always find a maximal isotopy $I$ such that $F\subset \textrm{fix}(I)$. Since the lift of the restriction of $f$ to $\textrm{dom}(I)$ is a Brouwer homeomorphism commuting with all covering automorphisms, Theorem \ref{teofolheacao} tell us that one may find a foliation $\widetilde\F$ by Brouwer lines that descends to an oriented nonsingular foliation of $\textrm{dom}(I)$. It is therefore possible to define a singular foliation $\F$ of $M$ such that $\textrm{dom}(\F)=\textrm{dom}(I)$, and such that $F$ is contained in $\textrm{sing}(\F)$. 
Furthermore, one can show (\cite{le2005version}) that:
\begin{teo}\label{foltal}
Given a maximal isotopy $I$, there exists $\F$ a oriented topological foliation with singularities of $M$, with $\emph{dom}(\F)=\emph{dom}(I)$, such that for all $z\in\emph{dom}(I)$ the trajectory $I(z)$ is homotopic with fixed endpoints in $\emph{dom}(I)$ to a path transverse to $\F$, and this path is unique up to equivalence. 
\end{teo}

We denote by $I_\F(z)$ to the class of paths described in the previous theorem, as well as to its representatives whenever the context is clear.  We further denote $$I^n _\F(z)=\Pi_{0\leq k<n}I_\F(f^k(z)), \,\,\,I^\N_\F(z)=\Pi_{k\geq 0}I_\F(f^k(z)) \,\textrm{ e }\,I^\Z_\F(z)=\Pi_{k\in\Z}I_\F(f^k(z)).$$

\begin{definicao}
A transversal path $\gamma:[a,b]\to\textrm{dom(I)}$ is called \textit{admissible of order $n$} if there exists $z\in\textrm{dom}(I)$ such that $\gamma$ is equivalent to a path in $I_\F^n(z)$. A path that is admissible of some order is just called admissible.
\end{definicao}

Note that Proposition 19 of \cite{LeCalvezTal2015ForcingTheory} has as a direct consequence that:

\begin{lema}\label{lemasubcaminhosadmissiveis}
Let $\beta, \gamma:[a,b]\to\textrm{dom(I)}$ be paths transversal to $\F$, such that $\beta\pitchfork_{\F}\gamma$. If $\gamma$ is admissible of order $n$ and $I\subset[a,b]$ is a nondegenerate interval, then $\gamma\mid_{I}$ is also admissible of order $n$.
\end{lema}

\subsection{Forcing}

Let us present now some of the results from the forcing theory developed by Le Calvez and the second author. As before we assume fixed a homeomorphism $f$ of the surface $M$, $I$ a maximal isotopy joining the identity and $f$, and $\F$ a Brouwer-Le Calvez foliation for $I$. The fundamental lemma is the following:

\begin{proposicao}[\cite{LeCalvezTal2015ForcingTheory}]
\label{propadm}
Suppose $\gamma_i:[a_i,b_i]\to M$, $i=1,2$ are two transverse paths that intersect $\F$-transversally at $\gamma_1(t_1)=\gamma_2(t_2)$. If $\gamma_1$ is admissible of order $n_1$ and $\gamma_2$ is admissible of order $n_2$, then both the transverse paths $\gamma_1|_{[a_1,t_1]}\gamma_2|_{[t_2,b_2]}$ and $\gamma_2|_{[a_2,t_2]}\gamma_1|_{[t_1,b_1]}$ are admissible of order $n_1+n_2$. Furthermore, either one of these paths is admissible of order $\min(n_1,n_2)$ or both are admissible of order $\max(n_1,n_2)$. 
\end{proposicao}

\begin{definicao}\label{defdirigido}
Let $\gamma:\R\to\R^2$ be a proper path and let  $\rho\in\R^2$, with $||\rho||=1$. We say that $\gamma$ é \emph{directed by $\rho$} if $$\lim_{t\to\pm\infty}||\gamma(t)||=+\infty,\quad\lim_{t\to+\infty}\gamma(t)/||\gamma(t)||=\rho,\quad\lim_{t\to-\infty}\gamma(t)/||\gamma(t)||=-\rho .$$
\end{definicao}

We note that whenever $f:\T\to\T$ is a homeomorphism homotopic to the identity and $z\in\T$ is a point whose rotation vector is well defined and equal to $\rho\not=0$, then we can find $\gamma$ a representative of $I^\Z_\F(z)$, such that every lift of $\gamma$ to  $\R^2$ (the universal covering of $\T$) is a path directed by $\rho/||\rho||$.

Assume $M$ is $\R^2$,  and let $\gamma:\R\to\R^2$ be a transverse line. Denote, as before, the two connected components of its complement by $R(\gamma)$ and $L(\gamma)$. We define the \emph{foliated right} (resp. \emph{foliated left} of $\gamma$,  denoted as $r(\gamma)$ (resp. $l(\gamma)$) as the set of leafs and singularities of $\F$ stictly contained in $R(\gamma)$ ( resp. $L(\gamma)$). Thus $l(\gamma)\cup r(\gamma)$ contains all leafs and singularities not intersected by $\gamma$. The following is a useful criterium to detect transversal intersections:

\begin{proposicao}
\label{proprl}
Let $\F$ be an oriented topological foliation with singularities of $\R^2$,  $\gamma:\R\to\R^2$ be a transverse path where $\gamma$  is a line , and let $\gamma':[a,b]\to\R^2$ be a transverse path. If $[\gamma']\cap l(\gamma)\neq\emptyset$, $[\gamma']\cap r(\gamma)\neq\emptyset$ and if $J^*=\{t\in\R\mid\phi_{\gamma(t)}\textrm{ crosses }\gamma'\}$ is bounded, then there exist intervals $J$ and $J'$ such that $\gamma'|_{J'}\pitchfork_{\F}\gamma|_J$.
\end{proposicao}
\begin{proof}
Let $t_0, s_0$ in $[a,b]$ be such that $\gamma'(t_0)\in l(\gamma)$ and $\gamma'(s_0)\in r(\gamma)$, and we assume, without loss of generality, that $t_0<s_0$. Since $l(\gamma)$ and $r(\gamma)$ are closed, let $t_1=\max_{t_0\le t <s_0}\{t\mid\gamma'(t)\in l(\gamma)\}$ and let $s_1= \min_{t_1<t\le s_0}\{t\mid\gamma'(t)\in r(\gamma)\}$. Let $J'=[t_1,s_1]$, and note that  $[\gamma'|_{J'}]$ must intersect $\gamma$, so let $t'$ be such that $\gamma'(t')\in [\gamma]$. Also, for $t_1<t<s_1$, $\gamma'(t)$ is in a leaf that crosses $\gamma$ and so there exists $a_0, b_0$ such that $\gamma'|_{(t_1, s_1)}$ is equivalent to $\gamma|_{(a_0,b_0)}$, where $a_0$ and $b_0$ are finite since both are contained in $J^*$. Set $J=[a_0, b_0]$ and let $s$ be such that $\gamma(s)=\gamma'(t')$. Note that, by lifting $\gamma'|_{J'}$ and $\gamma|_J$ to paths $\tilde\gamma'$ and $\tilde\gamma$ in the universal covering of $\emph{dom}(\F)$ such that $\tilde\gamma(s)=\tilde\gamma'(t')$,  and lifting $\F$ to $\tilde\F$, one has that the leaf $\phi_0=\phi_{\tilde\gamma'(t')}$ is such that $\tilde\gamma(a_0)$ is above $\tilde\gamma'(t_1)$ and $\tilde\gamma(b_0)$ is below $\tilde\gamma'(s_1)$. Thefore $\tilde\gamma'\pitchfork_{\tilde\F}\tilde\gamma$ and so $\gamma'|_{J'}\pitchfork_{\F}\gamma|_J$.
\end{proof}

The next result shows that admissible paths obey some sort of continuity:

\begin{lema}[\cite{LeCalvezTal2015ForcingTheory}]
\label{lemaestabilidade}
Fix $z\in\emph{dom}(I)$, $n\geq 1$, and let us parametrize $I^n_{\F}(z)$ by $[0,1]$. For each $0<a<b<1$, there exists a neighborhood $V$ of $z$ such that, for all $z'\in V$,  $I^n_{\F}(z)|_{[a,b]}$ is equivalent to a subpath of $I^n_{\F}(z')$. Furthermore, There exists $W$ a neighborhood of $z$ such that, for all $z',z''\in W$, the path $I^n_{\F}(z')$ is equivalent to a subpath of $I^{n+2}_{\F}(f^{-1}(z''))$.
\end{lema}

One key fact used on the proof of the main theorem of this work is that typical points for ergodic measures are recurrent. We present a similar definition for transverse paths:

\begin{definicao}\label{defrecorrente}
A transverse path $\gamma:\R\to M$ is called \emph{$\F$-recurrent} if for every compact $J\subset\R$ and all $t\in\R$ there exists segments $J'\subset(-\infty,t]$ and $J''\subset[t,+\infty)$ such that $\gamma|_{J'}\sim_{\F}\gamma|_{J}$ e $\gamma|_{J''}\sim_{\F}\gamma|_{J}$.
\end{definicao}

It follows from \ref{lemaestabilidade} that:
\begin{corolario}
\label{correcorrente}
If $z\in\emph{dom}(I)$ is a recurrent point for $f$, then $I^\Z_{\F}(z)$ is  $\F$-recurrent.
\end{corolario}

We also need the following technical proposition:

\begin{proposicao}[\cite{LeCalvezTal2015ForcingTheory}]
\label{proppqeps}
Let $f:\T\to\T$ be an homeomorphism isotopic to the identity, $\hat{f}:\R^2\to\R^2$ a lift of $f$ and assume that $\rho(\hat f)=\{t\rho_0\mid 0\leq t\leq 1\}$, where $\tan(\rho_0)\notin\mathbb{Q}$. Then there exists $\hat y_0\in\emph{dom}(\hat\F)$ such that for every $\epsilon\in\{-1,1\}$ there exists a sequence $(p_l,q_l)_{l\geq 0}$ in $\Z^2\times\N$ satisfying: $$\lim_{l\to+\infty}q_l=+\infty, \quad\lim_{l\to+\infty}\tilde{f}^{q_l}(\tilde y_0)-\tilde y_0-p_l=0, \quad\epsilon\langle p_l,\rho_0^\perp\rangle>0$$ and a sequence $(p'_l,q'_l)_{l\geq 0}$ in $\Z^2\times\N$ satisfying: $$\lim_{l\to+\infty}q'_l=+\infty, \quad\lim_{l\to+\infty}\tilde{f}^{-q'_l}(\tilde y_0)-\tilde y_0-p'_l=0, \quad\epsilon\langle p'_l,\rho_0^\perp\rangle>0.$$ Furthermore, we can take $y_0=\pi(\tilde{y_0})$ such that $\rho(\tilde f,y_0)=\rho_0$.
\end{proposicao}

We finish with one of the main results from \cite{calvez2018topological}.

\begin{teo}
\label{teoauto}
Let $M$ be an oriented surface, $f$ an homeomorphism of $M$ isotopic to the identity, $I$ a maximal isotopy for $f$ and $\F$ a Brouwer-Le Calvez foliation transverse to $I$. Assume that  $\gamma:[a,b]\to\emph{dom}(I)$ is an admissible path of order $r$ with a transverse self-intersectin at $\gamma(s)=\gamma(t)$, where $s<t$. Let $\tilde\gamma$ be a lift of $\gamma$ to the universal covering $\widetilde{\emph{dom}}(I)$ of $\emph{dom}(I)$ and let $T$ be a covering transformation such that $\tilde\gamma$ and $T(\tilde\gamma)$ have a $\tilde\F$-transverse intersection at $\tilde\gamma(t)=T(\tilde\gamma)(s)$. Let $\tilde f$ be the lift of $f|_{\emph{dom}(I)}$ to $\widetilde{\emph{dom}}(I)$, and $\hat f$ be the homeomorphism of the anular covering space $\widehat{\emph{dom}}(I)=\widetilde{\emph{dom}}(I)/T$ that is lifted by $\tilde f$. Then:
\begin{enumerate}[(i)]
\item For every rational number $p/q\in(0,1]$, writen in its irreducible form, there exists $\tilde z\in\widetilde{\emph{dom}}(I)$ such that $\tilde{f}^{qr}(\tilde z)=T^p(\tilde z)$ and $\tilde{I}_{\tilde\F}^{\Z}(\tilde z)$ is equivalent to $\Pi_{k\in\Z}T^k(\tilde\gamma|_{[s,t]})$;
\item For every irrational number $\lambda\in[0,1/r]$, there exists a compact $\hat f$ invariant set $\hat{Z}_\rho\subset\widehat{\emph{dom}}(I)$ such that every point$\hat z\in\hat{Z}_\rho$ has a rotation number $\rho(\tilde f,\hat z)=\lambda$. Furthermore, if $\tilde z$ is a lift of $\hat z$, then $\tilde{I}^{\Z}_{\tilde\F}(\tilde z)$ is equivalent to $\Pi_{k\in\Z}T^k(\tilde\gamma|_{[s,t]})$.
\end{enumerate}
\end{teo}

\section{Proofs of the main results}\label{sec:mainresult}
\subsection{Proof of Theorem A}
Let us fix an isotopy $I:\T\times[0,1]\to\T$ between $f$ and the identity, such that $\textrm{fix}(I)\neq\emptyset$, and $\tilde{I}:\R^2\times[0,1]\to\R^2$ a lift of $I$ such that $\textrm{fix}(\tilde I)\neq\emptyset$. By Theorem \ref{teobeguin2016fixed}, we may assume that $I$ is maximal. Let $\F$ be the foliation of $\T$ given by Theorem \ref{foltal}, and $\tilde\F$ its lift for $\R^2$. Let also $\mu_0$ be an ergodic measure whose the rotation vector is $\rho_0$. Let us further assume that $\rho_0$ is in the first quadrant, i.e., that both $(\rho_0)_1$ and $(\rho_0)_2$ are positive, the other cases are analogous.

\begin{lema}\label{lemacurvas}
There are  $\tilde\F$-transverse lines $\alpha_-,\alpha_+:\R\to\R^2$ and $v_-,v_+\in\Z^2$ such that $\alpha_-(t+1)=\alpha_-(t)+v_-$, $\alpha_+(t+1)=\alpha_+(t)+v_+$, for every $t\in\R$, and such that $\langle v_-,\rho_0^\perp\rangle<0$, $\langle v_-,\rho_0\rangle>0$, $\langle v_+,\rho_0^\perp\rangle>0$ and $\langle v_+,\rho_0\rangle>0$.
\end{lema}
\begin{proof}
Let us start by building $\alpha_+$ and $v_+$. Let $\tilde y_0$ and $(p_l,q_l)_{l\geq 0}$  given by Proposition \ref{proppqeps} such that $\lim_{l\to+\infty}q_l=+\infty$, $\lim_{l\to+\infty}(\tilde{f}^{q_l}(\tilde y_0)-\tilde y_0-p_l)=0$ and $\langle p_l,\rho_0^\perp\rangle>0$. Let us fix $\alpha_0\in\tilde{I}_{\tilde\F}(\tilde y_0)$. Let $V_0\subset\R^2$ be a small neighborhood of $\tilde y_0$ such that there exists $h_0:V_0\to(0,1)^2$ a homeomorphism mapping the restriction of $\tilde\F$ to $V_0$ into the vertical foliation oriented downwards in $(0,1)^2$ and such that $(h_0([\alpha_0]\cap V_0))_1=[(h_0(\tilde y_0))_1,1)$ (i.e., $h(\alpha_0\cap V_0)$ crosses all the leaves  on the left of the leaf that passes through $h_0(\tilde y_0)$). Let also  $\varepsilon>0$ such that $B(\varepsilon,\tilde y_0)\subset V_0$. As noted in the Proposition \ref{proppqeps} we can assume that $\rho(\tilde f,\tilde y_0)=\rho_0$, so there is $l_0\in\N$ such that for $l\geq l_0$  we have $\langle p_l,\rho_0\rangle>0$. By Proposition \ref{proppqeps} we have $\lim_{l\to+\infty}(\tilde{f}^{q_l}(\tilde y_0)-\tilde y_0-p_l)=0$, so there is $l_1>0$ such that $\tilde{f}^{q_l}(\tilde y_0)-p_{l}\in B(\varepsilon,\tilde y_0)$ and $\langle p_{l},\rho_0^\perp\rangle>0$, for all $l\geq l_1$. Making $l'=\max\{l_0,l_1\}$, let us denote $N=q_{l'}$ and $v_+=p_{l'}$.

\begin{figure} 
\centering
\includegraphics[scale = 0.5]{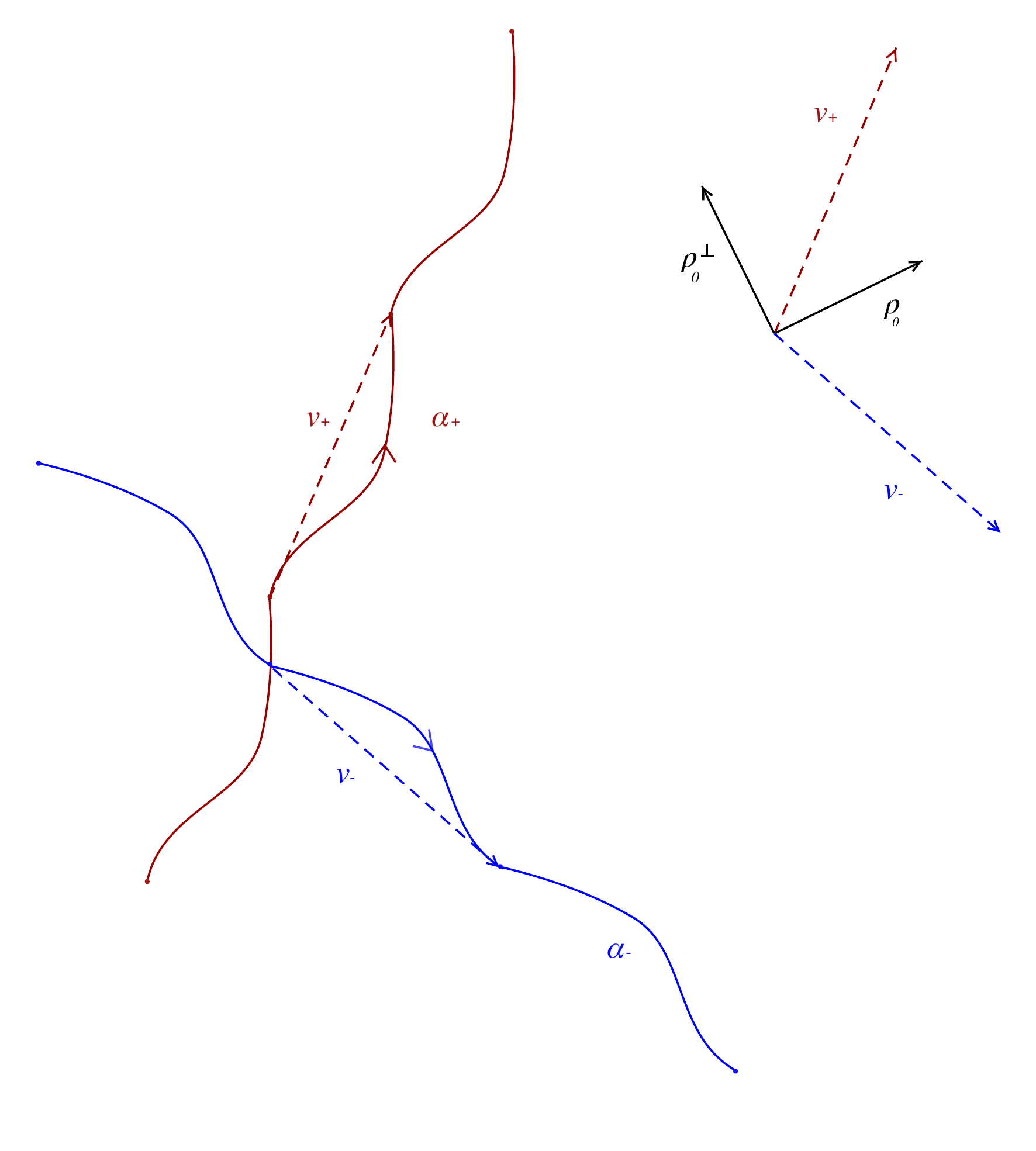}
\caption{Illustration of the Lemma \ref{lemacurvas}}
\label{figalpha}
\end{figure}

Consider now $\alpha_N\in\tilde{I}^N_{\tilde\F}(\tilde y_0)$ such that $[\alpha_0]\subset[\alpha_N]$ and $\alpha_N$ is parameterized by $[0,1]$. Since $\tilde{f}^N(\tilde y_0)-v_+\in B(\varepsilon,\tilde y_0)$, we have that $\tilde{f}^N(\tilde y_0)-v_+\in V_0$, therefore we can modify $\alpha_N$ inside $V_0$ in order to obtain a transverse path $\alpha_N':[0,1]\to\R^2$ so that $\alpha_N$ and $\alpha_N'$ are equal outside of $V_0$ and $\alpha_N'(0)=\tilde{f}^N(\tilde y_0)-v_+$ (to modify $\alpha_N$ in $V_0$, it is enough to modify $h_0(\alpha_N\cap V_0)$ in $(0,1)^2$ and take it back to $V_0$ using $h_0^{-1}$, see Figure \ref{figlemacurvas}). Now, let us define $\alpha'_+=\Pi_{k\in\Z}(\alpha_N'+kv_+)$. Since $\alpha'_+$ is a transverse path, if $\alpha'_+$ has self-intersection, i.e., if $\alpha'_+(s)=\alpha'_+(t)$, with $s<t$, we can remove the arc $\alpha'_+|_{(s,t]}$  and reparametrize in a suitable way, getting like this a new path, which we shall denote by $\alpha_+$. Therefore we can assume that $\alpha_+$ is a simple path and so, by the construction, we get that $\alpha_+$ is a line satisfying the conditions of the statement.

\begin{figure} 
\centering
\includegraphics[scale = 0.9]{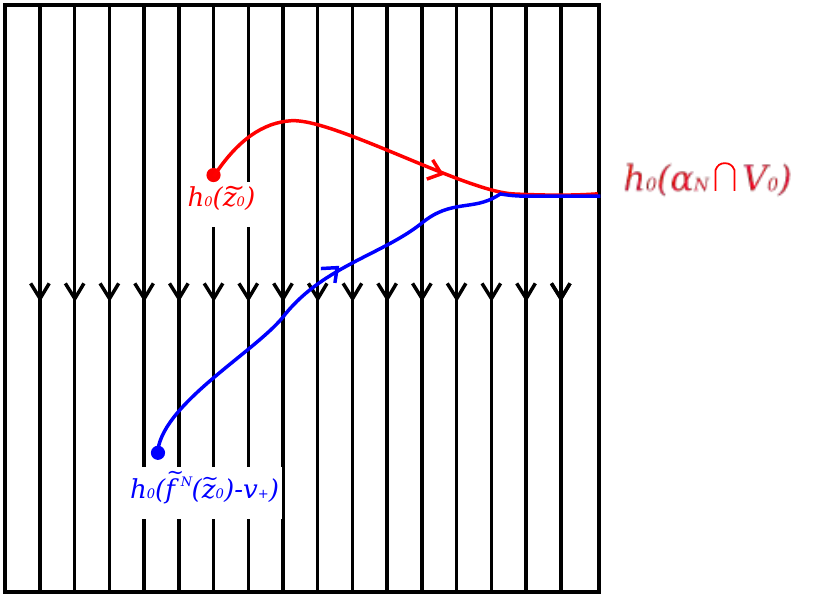}
\caption{Construction of $\alpha_N'$}
\label{figlemacurvas}
\end{figure}

The construction of $\alpha_-$ and $v_-$ is analogous, using  Proposition \ref{proppqeps} with $\epsilon=-1$.
\end{proof}

\begin{lema}\label{lemasobra}
There is $L_0>0$ such that, for every $\tilde{x}\in\R^2\setminus\emph{sing}(\tilde{\F})$, there is a transverse path $\tilde{\gamma}_{\tilde x}\in\tilde{I}_{\tilde{\F}}(\tilde x)$,  with $\emph{diam}(\tilde{\gamma}_{\tilde x})<L_0$. 
\end{lema}
\begin{proof}
First, let us note that it is enough to prove the result for points in $[0,1]^2$, since $\tilde{I}_{\tilde\F}(\tilde x+w)=\tilde{I}_{\tilde\F}(\tilde x)+w$ for every $\tilde x\in\R^2$ and every $w\in\Z^2$. Now, since $\tilde I$ is continuous, there is $L>0$ such that $\tilde{I}([0,1]^2\times [0,1])\subset B(L,0)$, i.e., for every point $\tilde x\in[0,1]^2$ the isotopy path $\tilde{I}(\tilde x)$ is contained in $B(L,0)$. 

Now, let  $\alpha_+$ and $\alpha_-$ be the transverse lines given by Lemma \ref{lemacurvas}, and also let $v_1,v_2,v_3,v_4\in\Z^2$ be such that $B(L,0)$ is contained in a bounded connected component of $\R^2\setminus([\alpha_++v_1]\cup[\alpha_-+v_2]\cup[\alpha_++v_3]\cup[\alpha_-+v_4])$, which we shall denote by $U$, and $U\subset R(\alpha_++v_1)\cap R(\alpha_-+v_2)\cap L(\alpha_++v_3)\cap L(\alpha_-+v_4)$ (see Figure \ref{figsobra1}). 

\begin{figure} 
\centering
\includegraphics[scale = 0.45]{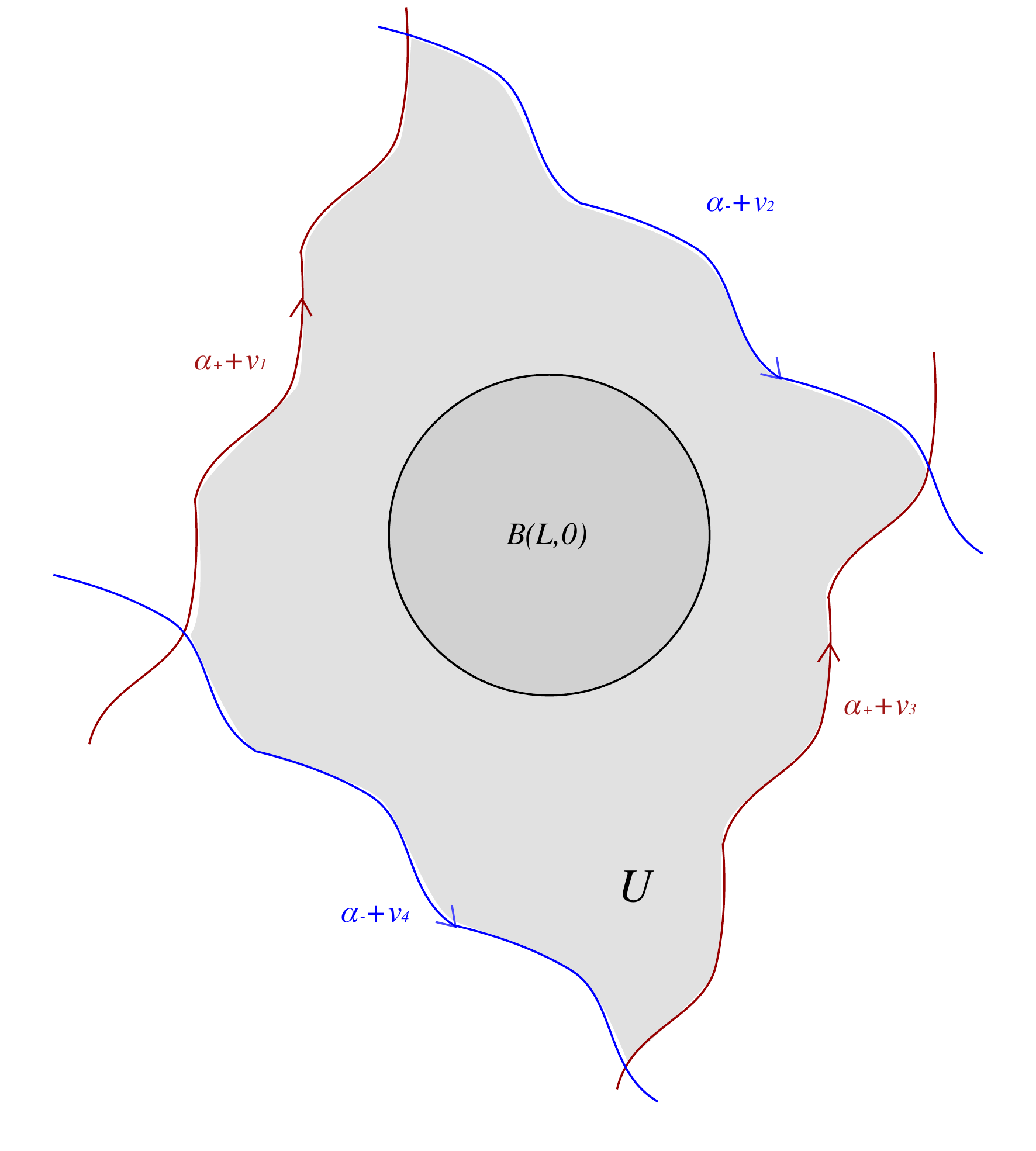}
\caption{Illustration of the set $U$}
\label{figsobra1}
\end{figure}

We claim that if $\phi$ is a leaf of $\tilde\F$ that intersects $U$, then $[\phi]\cap U$ is connected (that is, a segment of $\phi$). In fact, as remarked right after Definition \ref{defequivalence}, since $\alpha_+$ and $\alpha_-$ (as well as their translations) are transverse lines, we have that the leaf $\phi$ crosses each line at most once, and always from left to right of the line. Now let $t_1<t_2$ be such that $\phi(t_i)\in U,\, i\in\{1,2\}$. Note that to prove that $[\phi]\cap U$ is connected, it is enough to prove that $\phi(t)\in U$, for every $t_1<t<t_2$. Since $\phi(t_1)\in U$ and $t>t_1$, we have that $\phi(t)\in R(\alpha_++v_1)$ and $\phi(t)\in R(\alpha_-+v_2)$. Analogously, since $\phi(t_2)\in U$ and $t<t_2$, we have that $\phi(t)\in L(\alpha_++v_3)$ and $\phi(t)\in L(\alpha_-+v_4)$. So, $\phi(t)\in R(\alpha_++v_1)\cap R(\alpha_-+v_2)\cap L(\alpha_++v_3)\cap L(\alpha_-+v_4)$. Furthermore, as $\phi(t_1)$ and $\phi(t_2)$ belong to $U$, we have that no point of $\phi|_{[t_1,t_2]}$ intersects $[\alpha_++v_1]\cup[\alpha_-+v_2]\cup[\alpha_++v_3]\cup[\alpha_-+v_4]$. Therefore $\phi(t)\in U$, proving that $[\phi]\cap U$ is connected.

Let $\widehat{\textrm{dom}}(\tilde I)$ be the universal covering of $\textrm{dom}(\tilde I)$, $\hat\pi:\widehat{\textrm{dom}}(\tilde I)\to\textrm{dom}(\tilde I)$ the covering map, $\widehat{I}$ a lift of $\tilde{I}|_{\textrm{dom}(\tilde I)}$ to $\widehat{\textrm{dom}}(\tilde I)$ and $\hat\F$ the lift of $\tilde\F$ to $\widehat{\textrm{dom}}(\tilde I)$. Since $I$ is maximal, we have that $\tilde I$ is maximal, and therefore $\hat f:\widehat{\textrm{dom}}(\tilde I)\to\widehat{\textrm{dom}}(\tilde I)$  is a Brouwer homeomorphism, where $\hat f$ is a lift of $\tilde{f}|_{\textrm{dom}(\tilde I)}$. Let us fix now $\tilde x\in[0,1]^2$ such that $\tilde x\notin \textrm{sing}(\tilde\F)$ and $\hat{x}\in\widehat{\textrm{dom}}(\tilde I)$ a lift of $\tilde x$. Being $\Phi(\hat x)=\{\phi\in\hat\F|\,\hat{x}\in R(\phi) \textrm{ and }\hat{f}(\hat x)\in L(\phi)\}\cup\{\phi_{\hat x},\phi_{\hat{f}(\hat x)}\}$, note that $\Phi(\hat x)$ is the set of leafs that intersect the transverse trajectory $\hat{I}_{\hat\F}(\hat x)$, and beyond that, $\Phi(\hat x)$ is totally ordered by the relation $\phi_1<\phi_2$ if $R(\phi_1)\subset R(\phi_2)$, because $\hat f$ is a Brouwer homeomorphism and the leaves of $\hat\F$ are Brouwer lines. Using this order, we can parameterize $\Phi(\hat x)$ by a parameter $s\in[0,1]$ in such a way that $\phi_0=\phi_{\hat x}$ and $\phi_1=\phi_{\hat{f}(\hat x)}$.

Denoting the isotopy path  $\hat{I}(\hat x):[0,1]\to\widehat{\textrm{dom}}(\tilde I)$ by $\hat\beta$, we can define the following functions of the parameter $s$
$$t_-(s)=\begin{cases}
               0, \,\textrm{if }s=0\\
               \inf\{t\in[0,1]\mid\hat{\beta}([t,1])\cap R(\phi_s)=\emptyset\}, \,\textrm{if }s\in(0,1]
            \end{cases} $$
and
$$t_+(s)=\begin{cases}
               \inf \{t\in[0,1]\mid \hat{\beta}([t, 1])\subset L(\phi_s)\}, \,\textrm{if }s\in[0,1)\\
               1, \,\textrm{if }s=1.
            \end{cases} $$
Intuitively, $t_-(s)$ is the moment when $\hat\beta$ was on the right side of $\phi_s$ for the last time, and $t_+(s)$ is the first moment in which $\hat\beta$ is always on the left side of $\phi_s$. Note that, if $s_1<s_2$, then $t_{-}(s_1)\leq t_{+}(s_1) < t_{-}(s_2)\leq t_{+}(s_2)$. So we have that $t_-$ and $t_+$ coincide, except possibly in a countable set of discontinuities. However, note that if we list the discontinuity points as $(s_i)_{i\in\N}$, we have that $\sum d_i\leq 1$, where $d_i=t_+(s_i)-t_-(s_i)$ (because $t_\pm([0,1])\subset[0,1]$).

Now let us define a path $\gamma^*_{\hat x}:[0,1]\to\widehat{\textrm{dom}}(\tilde I)$ which will be transverse except in the discontinuities of $t_-$ and $t_+$ as follows: we make $\gamma_{\hat x}^*$  be equal to $\hat\beta$ in the points where $t_-(s)=t_+(s)$ and where $t_-(s)\neq t_+(s)$ we make  $\gamma_{\hat x}^*$ be equal to the leaf segment $\phi_s$ which connects $\hat{\beta}(t_-(s))$ to $\hat{\beta}(t_+(s))$ (see Figure \ref{figsobra2}).

\begin{figure} 
\centering
\includegraphics[scale = 0.75]{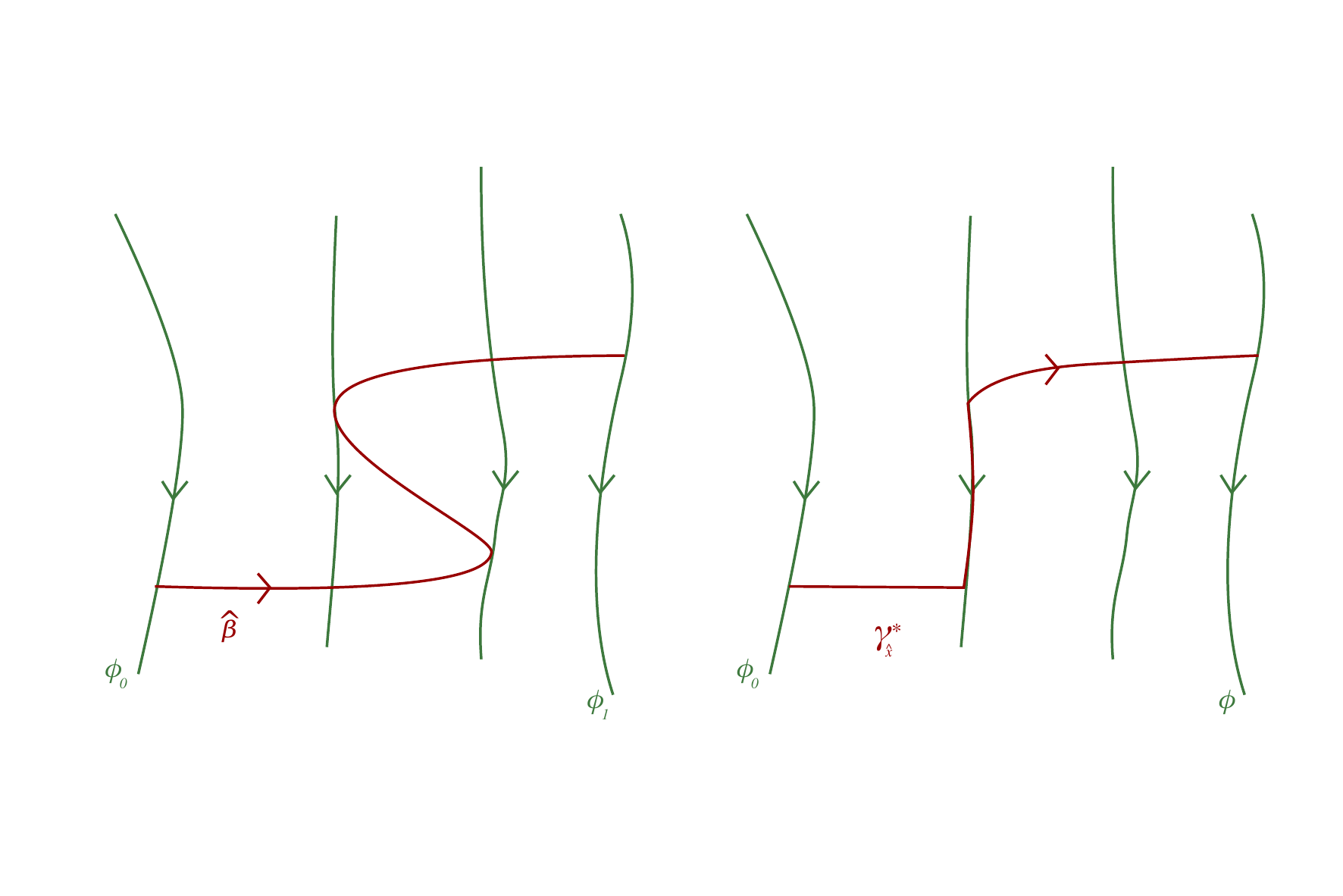}
\caption{Illustration of the curves $\hat\beta$ e $\gamma^*_{\hat x}$}
\label{figsobra2}
\end{figure}

Now, since $\gamma^*_{\hat x}$ is made by leaves points or leaves segments, we have that for each $s\in[0,1]$ we can find $V_s$ a tubular neighborhood of the point (or segment) of $\gamma^*_{\hat x}$ which intersects the leaf $\phi_s$ and $\varepsilon_s\in(0,1)$ so that $V_s\subset B(\varepsilon_s,[\gamma_{\hat x}^*]\cap[\phi_s])$. Then, $[\gamma^*_{\hat x}]\subset\cup_{s\in[0,1]}V_s$, and by compactness we have $[\gamma^*_{\hat x}]\subset\cup_{j=0}^k V_{s_j}\subset\cup_{j=0}^k B(\varepsilon_{s_j},[\gamma_{\hat x}^*]\cap[\phi_{s_j}])$, for some $k\in\N$ (note that the diameter of $B(\varepsilon_{s_j},[\gamma_{\hat x}^*]\cap[\phi_{s_j}])$ is uniformly bounded, even for the values of $s_j$ such that $[\gamma_{\hat x}^*]\cap[\phi_{s_j}]$ is a leaf segment, because $\sum d_i\leq 1$). Therefore, we can partition the interval $[0,1]$ in a finite number of closed subintervals and modify $\gamma_{\hat x}^*$ in each interval, inside the tubular neighborhoods, keeping the ends of the intervals fixed, in order to obtain a transverse path $\gamma_{\hat x}$ such that $[\gamma_{\hat x}]\subset\cup_{j=0}^k B(\varepsilon_{s_j},[\gamma_{\hat x}^*]\cap[\phi_{s_j}])$ (see Figure \ref{figsobra3})                .

\begin{figure} 
\centering
\includegraphics[scale = 0.75]{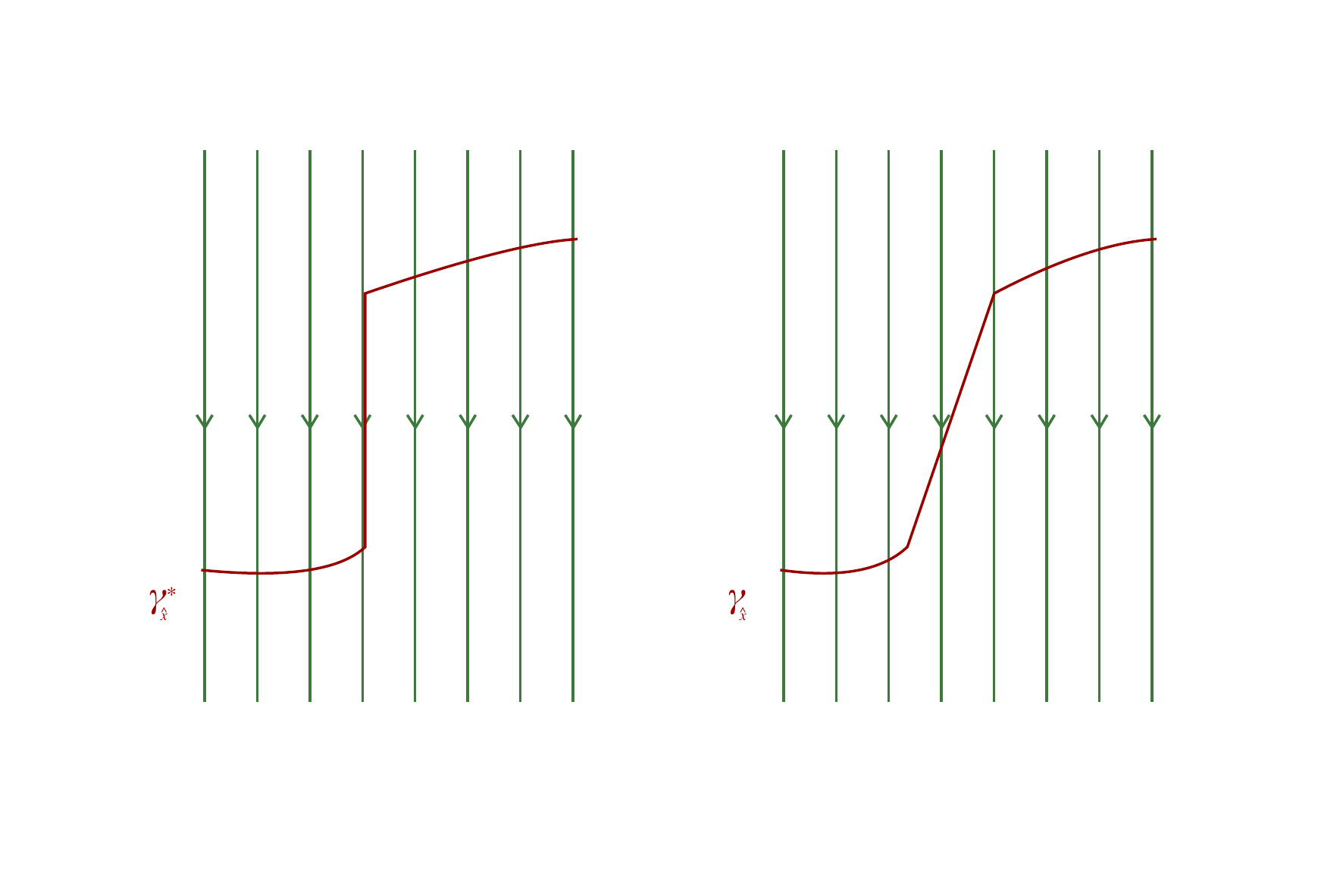}
\caption{Modification of $\gamma^*_{\hat x}$ in $\gamma_{\hat x}$}
\label{figsobra3}
\end{figure}

Note now that, since $\hat\pi(\hat\beta)=\tilde{I}(\tilde x)$ is contained in $U$ and, if $\phi$ is a leaf of $\tilde\F$, then $[\phi]\cap U$ is connected, therefore $\hat\pi(\gamma^*_{\hat x})$ is also contained in $U$. So, $\hat\pi(\cup_{j=0}^k B(\varepsilon_{s_j},[\gamma_{\hat x}^*]\cap[\phi_{s_j}]))\subset B(\max\{\varepsilon_{s_j}\},U)\subset B(1,U)$, and then, denoting $\hat\pi(\gamma_{\hat x})=\tilde{\gamma}_{\tilde x}$, we have that $\tilde{\gamma}_{\tilde x}\subset B(1,U)$ and $\tilde{\gamma}_{\tilde x}\in \tilde{I}_{\tilde\F}(\tilde x)$. Since $U$ is bounded, we have that there is $L_0>0$ such that $\textrm{diam}(\tilde{\gamma}_{\tilde x})<L_0$, proving the result.
\end{proof}

For now on, let us take $\tilde z_0$ to be the point given by proposition \ref{proppqeps} and let $z_0=\pi(\tilde z_0)$. Note that $z_0$ is recurrent, and has rotation vector $\rho_0$.
We will denote by $\g$ a element of $\tilde{I}_{\tilde{\F}}^\Z(\tilde z_0)$ which passes through $\tilde z_0$ and by $[\g]$ its image. Using Lemma \ref{lemasobra}, we can assume that for each $n\in\Z$, the segment of $\g$ between $\tilde{f}^n(\tilde{z}_0)$ and $\tilde{f}^{n+1}(\tilde{z}_0)$ has diameter less than $L_0$. We can also assume that $\g$ is parameterized so that $\g(n)=\tilde{f}^n(\tilde z_0)$, for every $n\in\Z$.

\begin{definicao}\label{defcone}
Let $v\in\R^2$ be a unit vector such that $v\neq \rho_0$ and $v\neq\rho_0^\perp$, and denote by $v_s$ the vector symmetrical to $v$ with respect to the direction of $\rho_0$ (i.e., $\langle v_s,\rho_0 \rangle=\langle v,\rho_0 \rangle$ and $\langle v_s,\rho^\perp_0 \rangle=-\langle v,\rho^\perp_0 \rangle$). Let us now denote the straight lines generated by $v$ and $v_s$ passing through $\tilde{z}_0$ by $r_v$ and $r_{v_s}$ (i.e., $r_v(t)=\tilde{z}_0+tv$, for $t\in\R$). We have that $\R^2\setminus([r_v]\cup [r_{v_s}])$ has four connected components, and denote by $C_1$ and $C_2$ the components that intersect the straight line generated by $\rho_0$ passing through  $\tilde{z}_0$. We will call \emph{cone generated by $\rho_0$ with inclination $v$ and origin $\tilde{z}_0$} to the closure of $C_1\cup C_2$, and we will denote such a set by $C_{\tilde{z}_0}^{\rho_0}(v)$ (see Figure \ref{figcone}).

\begin{figure} 
\centering
\includegraphics[scale = 0.5]{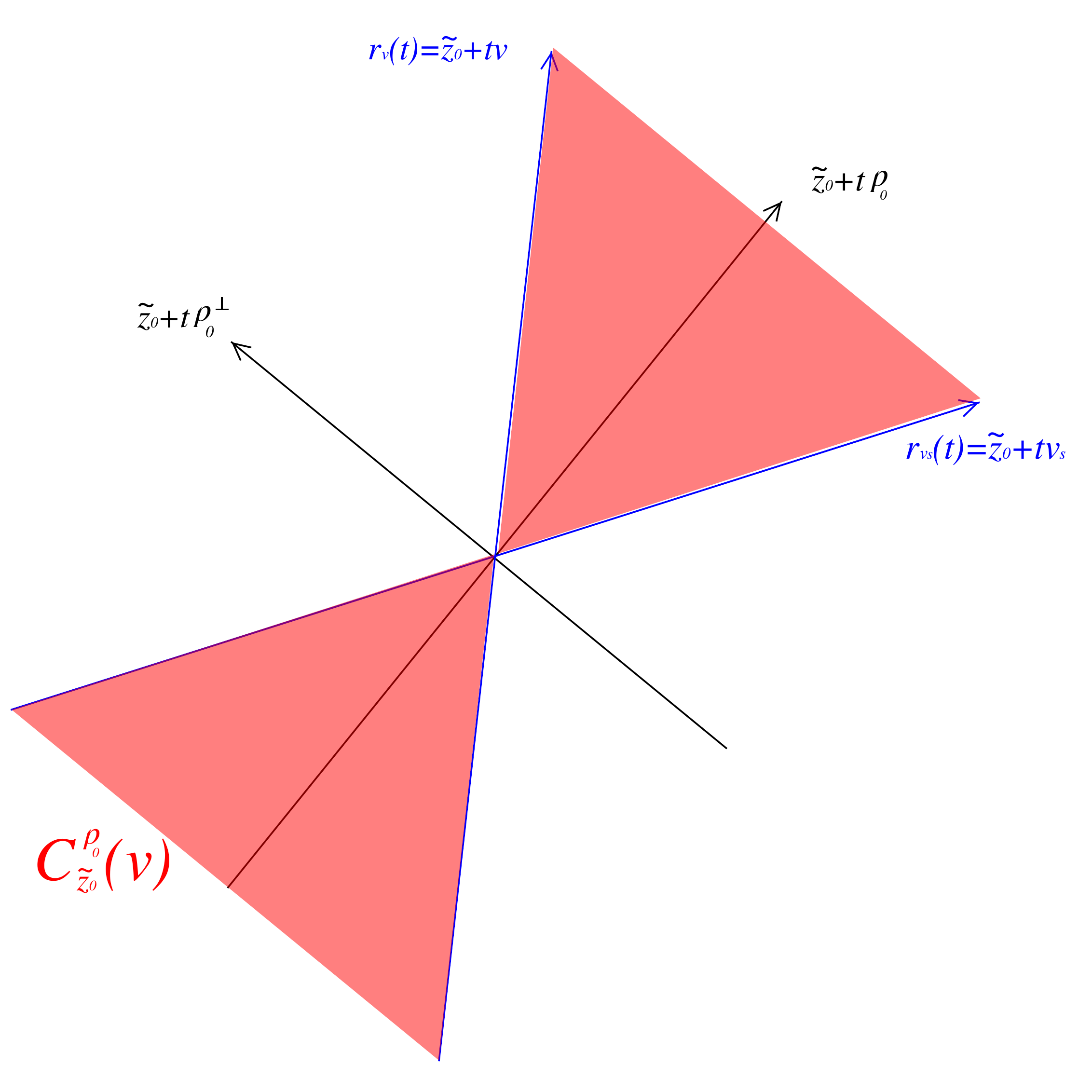}
\caption{Illustration of $C_{\tilde{z}_0}^{\rho_0}(v)$}
\label{figcone}
\end{figure}
\end{definicao}

Note that in Definition \ref{defcone} we have  $\partial(C_{\tilde{z}_0}^{\rho_0}(v))=[r_v]\cup [r_{v_s}]$

In the next lemma we will denote, for $L>0$,  $B(L,A)=\cup_{\tilde x\in A}B(L,\tilde x)$, where $A\subset\R^2$.

\begin{lema}\label{lemacone}
Given $v\in\R^2$ as in Definition \ref{defcone}, there is $L_1>0$  such that $[\g]\subset B(L_0,C_{\tilde{z}_0}^{\rho_0}(v))\cup B(L_1,\tilde{z}_0)$, where $L_0$ is given by the Lemma \ref{lemasobra}.
\end{lema}
\begin{proof}
First, let us consider the straight lines $r_v$ and $r_{v_s}$, as in Definition \ref{defcone}. Denoting by $d(r_v,x)$ the distance between the straight line $r_v$ and the point $\tilde x$, we have that $d_n=d(r_v, n\rho_0+\tilde{z}_0)=d(r_{v_s},n\rho_0+\tilde{z}_0)=n||\langle\rho_0,v\rangle v-\rho_0||$, for $n\in\Z$. Note that $B(d_n,n\rho_0+\tilde{z}_0)\subset C_{\tilde{z}_0}^{\rho_0}(v)$.

We have that $\lim_{n\to\infty}\frac{\tilde{f}^n(\tilde{z}_0)-\tilde{z}_0}{n}=\rho_0$. So, making $\varepsilon=||\langle\rho_0,v\rangle v-\rho_0||$, there is $n_1>0$ such that
\begin{align*}
d(\tilde{f}^n(\tilde{z}_0),n\rho_0+\tilde{z}_0)&=||\tilde{f}^n(\tilde{z}_0)-\tilde{z}_0-n\rho_0||\\ 
&< n\varepsilon=n||\langle\rho_0,v\rangle v-\rho_0||=d(r_v, n\rho_0+\tilde{z}_0), \quad\forall n\geq n_1
\end{align*}
Proceeding analogously for $\tilde{f}^{-1}$, we obtain 
\begin{align*}
d(\tilde{f}^{-n}(\tilde{z}_0),-n\rho_0+\tilde{z}_0)&=||\tilde{f}^{-n}(\tilde{z}_0)-\tilde{z}_0-n(-\rho_0)||\\
&< n\varepsilon=n||\langle\rho_0,v\rangle v-\rho_0||=d(r_v, n\rho_0+\tilde{z}_0), \quad\forall n\geq n_2.
\end{align*}
Therefore, setting $n_0=\max\{n_1,n_2\}$, we have that $d(\tilde{f}^{n}(\tilde{z}_0),n\rho_0+\tilde{z}_0)<d_n=d(r_v, n\rho_0+\tilde{z}_0)$, and since $B(d_n,n\rho_0+\tilde{z}_0)\subset C_{\tilde{z}_0}^{\rho_0}(v)$, we have $\tilde{f}^n(\tilde{z}_0)\in C_{\tilde{z}_0}^{\rho_0}(v)$, for $|n|\geq n_0$ (see Figure \ref{figlemacone}).

\begin{figure} 
\centering
\includegraphics[scale = 0.5]{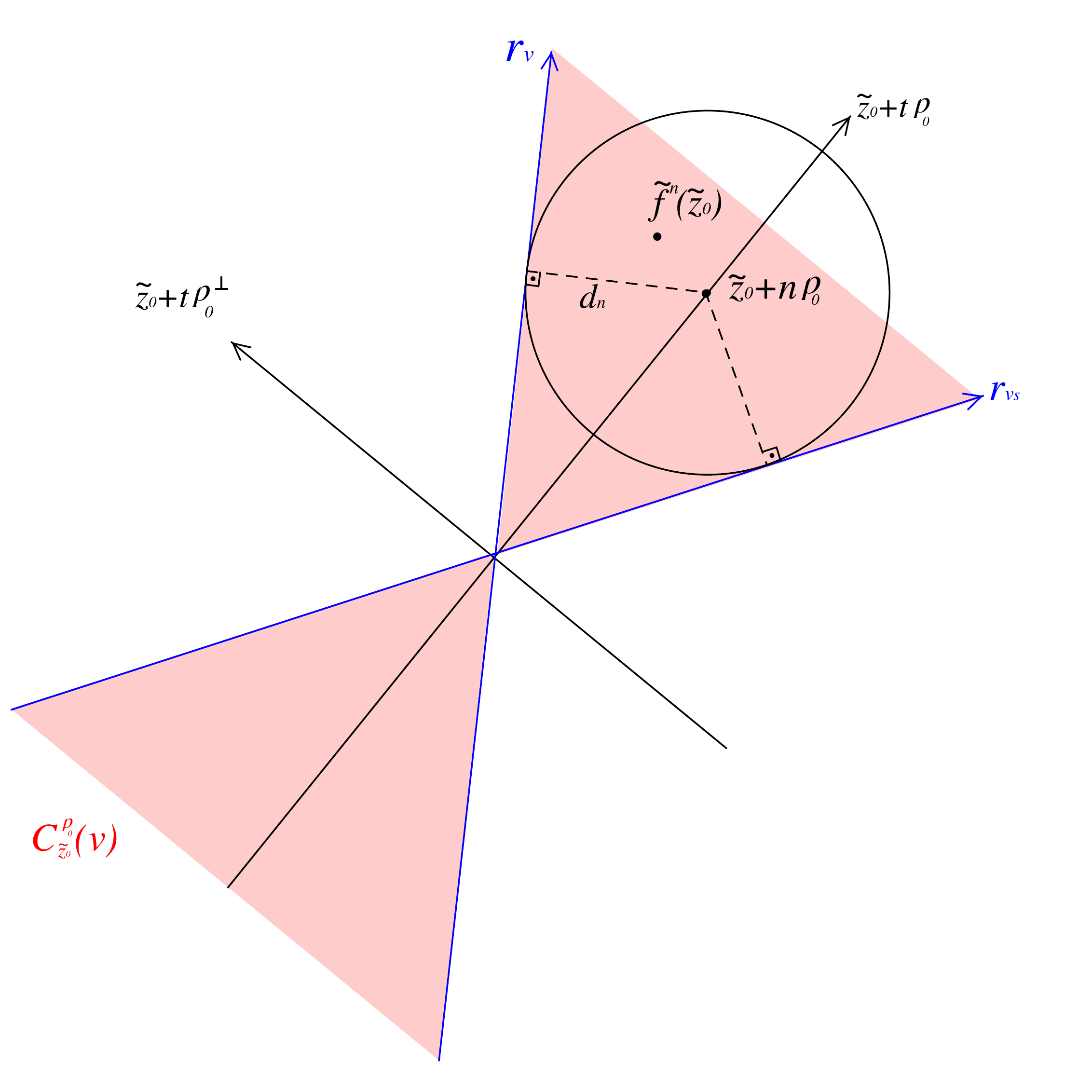}
\caption{Illustration of $\tilde{f}^n(\tilde{z}_0)\in C_{\tilde{z}_0}^{\rho_0}(v)$}
\label{figlemacone}
\end{figure}

By construction, we have that the diameter of the segment $\g$ between  $\tilde{f}^n(\tilde{z}_0)$ and $\tilde{f}^{n+1}(\tilde{z}_0)$ is smaller than $L_0$, so we have that for $|n|\geq n_0$ such segment is contained in $B(L_0,C_{\tilde{z}_0}^{\rho_0}(v))$. 

For $n$ such that $|n|<n_0$, $\tilde{f}^n(\tilde{z}_0)$ may be out of the cone, but since such points exist only in finite quantity, we have that there is $L'>0$ such that $||\tilde{f}^n(\tilde{z}_0)-\tilde{z}_0||<L'$, for $|n|<n_0$. Again, since each segment $\g$ has diameter bounded by $L_0$, we have that the segments of $\g$ between $\tilde{f}^n(\tilde{z}_0)$ and $\tilde{f}^{n+1}(\tilde{z}_0)$, with $|n|<n_0$, are contained in $B(L_1,\tilde{z}_0)$, where $L_1=L'+L_0$.
\end{proof}

\begin{lema}\label{lemaauto}
If $\tilde\alpha:[a,b]\to\R^2$ is an admissible path for $\tilde f$, there is no $w\in\Z^2_*$ such that $\tilde\alpha\pitchfork_{\tilde{\F}}(\tilde\alpha+w)$.
\end{lema}
\begin{proof}
Suppose there are $w\in\Z^2_*$ and $\tilde\alpha:[a,b]\to\R^2$ a path $r$-admissible such that $\alpha\pitchfork_{\tilde{\F}}(\tilde\alpha+w)$. Thus, we have by Theorem \ref{teoauto} that given $p/q\in(0,1]$ written in an irreducible way, $f$ will have a periodic point with rotation vector equal to $\frac{p}{qr}\cdot w$, which is an absurd, because $\rho(\tilde f)\cap\Q^2=\{(0,0)\}$.
\end{proof}

\begin{lema}\label{lemaautogamma}
There is no $w\in\Z^2$ such that $\g\pitchfork_{\tilde{\F}}(\g+w)$.
\end{lema}
\begin{proof}
If $w\neq 0$, the result is a direct consequence of Lemma \ref{lemaauto}. For $w=0$, let us suppose by contradiction that $\g$ has a transverse self-intersection. So, there are intervals $I,J\subset\R$ such that $\g|_I\pitchfork_{\tilde{\F}}\g|_J$. We can suppose that $I\subset(-N,N)$, for some $N\in\N$. Now let us note that since $z_0$ is recurrent, there are  $n_k\in\N$ and $w_k\in\Z^2$ such that $\tilde{f}^{n_k}(\tilde z_0)\to \tilde z_0+w_k$ or, equivalently, $\tilde{f}^{n_k}(\tilde z_0)-w_k\to\tilde z_0$, and then $\tilde{f}^{-N+n_k}(\tilde z_0)-w_k\to\tilde{f}^{-N}(\tilde z_0)$. Also, since $\rho(f,z_0)=\rho_0$, we have $w_k/n_k\to\rho_0$ and then there is $k_0$ such that $w_k\neq 0$ if $k>k_0$. 

Note that, by the way $\g$ has been parameterized, we have that $\g(-N)=\tilde{f}^{-N}(\tilde z_0)$ and $\g(N)=\tilde{f}^{N}(\tilde z_0)$, and then $\g\vert_{[-N,N]}=\tilde{I}_{\tilde\F}^{2N}(\tilde{f}^{-N}(\tilde z_0))$. Since $\tilde{f}^{-N+n_k}(\tilde z_0)-w_k\to\tilde{f}^{-N}(\tilde z_0)$, by Lemma \ref{lemaestabilidade} we have that if $k'$ is large enough, $\tilde{I}_{\tilde\F}^{2N}(\tilde{f}^{-N}(\tilde z_0))$ is equivalent to a sub-path of $\tilde{I}_{\tilde\F}^{2N+2}(\tilde{f}^{-1}(\tilde{f}^{-N+n_{k'}}(\tilde z_0)-w_{k'}))$. But 
\begin{align*}
\tilde{I}_{\tilde\F}^{2N+2}(\tilde{f}^{-1}(\tilde{f}^{-N+n_{k'}}(\tilde z_0)-w_{k'}))&=\tilde{I}_{\tilde\F}^{2N+2}(\tilde{f}^{-N-1+n_{k'}}(\tilde z_0)-w_{k'})\\
&=(\g-w_{k'})|_{[-N-1+n_{k'},N+1+n_{k'}]}
\end{align*}
and we can also assume that $k'>k_0$. So, there is $I'\subset[-N-1+n_{k'},N+1+n_{k'}]$ such that $(\g-w_{k'})|_{I'}\sim_{\tilde\F}\g|_{[-N,N]}$. Since $I\subset[-N,N]$, there is $I''\subset I'$ such that $(\g-w_{k'})|_{I''}\sim_{\tilde\F}\g|_I$, but  $\g|_I\pitchfork_{\tilde{\F}}\g|_J$, therefore $(\g-w_{k'})|_{I'}\pitchfork_{\tilde{\F}}\g|_J$, and given that $k'>k_0$, we have $w_{k'}\neq 0$, and thus we get a contradiction, from Lemma  \ref{lemaauto}.
\end{proof}

\begin{lema}
\label{lemalinhaunica}
$\g$ intersects each leaf at most once.
\end{lema}
\begin{proof}
Suppose, by contradiction, that there are $t'<t''$ such that $\phi_{\g(t')}=\phi_{\g(t'')}$. So we have that $\g|_{[t',t'']}$ is $\tilde\F$-equivalent to a transverse closed curve $\Gamma$. But since $\Gamma$ has some sub-path $\Gamma_0:J\to\R^2$ which is transverse, closed and simple, we have that $\g$ has a sub-path $\tilde\F$-equivalent to $\Gamma_0$. 
Since $\g$ does not have  transverse self-intersection, it follows from Proposition 20 of \cite{calvez2018topological} that:
\begin{enumerate}[(i)]
\item $\bigcup_{s\in J} [\phi_{\Gamma_0(s)}]=A_{\Gamma_0}$ is an open topological annulus;
\item $\{t\in\R\mid\g(t)\in A_{\Gamma_0}\}$ is an interval, which we denote by $I=(a,b)$, not necessarily bounded;
\item if $-\infty<a<b<+\infty$, then $\g(t')$ and $\g(t'')$ can not both belong to unbounded connected components of $\R^2\setminus A_{\Gamma_0}$.
\end{enumerate}
Since $z_0$ is recurrent and has rotation vector $\rho$ we have, as in the proof of the Lemma \ref{lemaautogamma}, that there are $t_k\to+\infty$, $w_k^+\in\Z^2$, with $||w_k^+||\to+\infty$ and $s_k\to-\infty$, $v_k^-\in\Z^2$, with $||v_k^-||\to+\infty$ such that $\g(t_k)\in A_{\Gamma_0}+w_k^+$ and $\g(s_k)\in A_{\Gamma_0}+v_k^-$.
Firstly, notice that, as $\Gamma$ is a transverse, closed and simple curve, every leaf $\phi$ that intersects $\Gamma_0$ it does so in only one point. From this it follows that either every leaf that intersects $\Gamma$ has its $\omega$-limit contained in the bounded connected component of its complement, or any leaf intersecting $\Gamma_0$ has its $\alpha$-limit contained in the bounded connected component of its complement. We will assume, without loss of generality, that the first situation occurs, and the second case is treated in the same way. But, since the foliation is invariant by integer translations, it follows that an analogous property holds for the leafs intersecting $\Gamma_0+w$, for integer vectors $w$. So, we have that, if $[\phi]\subset A_{\Gamma_0}$ and $[\Gamma_0+w]\cap[\Gamma_0]=\emptyset$, then we have that $[\phi]\cap A_{\Gamma_0}+w=\emptyset$, because the $\omega$-limit of a leaf at the intersection would be contained in two disjoint sets.

So, for $k$ large enough, since $[\Gamma_0+w_k^+]\cap[\Gamma_0]=\emptyset$, we have that $\phi_{\g(t_k)}$ is contained in an unbounded connected component of $\R^2\setminus A_{\Gamma_0}$, and the same holds for $\phi_{\g(s_k)}$. Therefore, by Proposition 20 of \cite{calvez2018topological}, we have that $\g$ has a transverse self-intersection, which is an absurd.
\end{proof}

\begin{corolario}
\label{corlinha}
$\g$ is a line.
\end{corolario}
\begin{proof}
By  Lemma \ref{lemalinhaunica} we have that $\g$ is a simple path, and since $||\tilde{f}^n(\tilde z_0)||\to+\infty$  when $n\to\pm\infty$, we have the result.
\end{proof}

Therefore, since $\g$ is a line, the sets $L(\g)$, $R(\g)$, $l(\g)$ and $r(\g)$ are well defined.

\begin{lema}\label{lemailimitada}
Let $A$ be a connected component of $r(\g)$ (or $l(\g)$). Then $A$ is an unbounded set.
\end{lema}
\begin{proof}
By Lemma \ref{lemalinhaunica} , we have that $\g$ is a  line and  crosses each leaf at most once. Thus, the set $\bigcup_{t\in\R}[\phi_{\g(t)}]$ of leaves that pass through $\g$ is homeomorphic to $\R^2$ and therefore simply connected. So, all connected components of the complement of such set are unbounded.
\end{proof}

\begin{lema}\label{lemacompacto}
If $K\subset\R^2$ is compact, then the set $I_K=\{t\in\R\mid[\phi_{\g(t)}]\cap K\neq\emptyset\}$ is also compact.
\end{lema}
\begin{proof}
Let  $\phi$ be a leaf such that $\phi(\bar s)\in K$, for some $\bar s\in\R$, and $V_-^w,V_+^w\in\Z^2$ such that $K\subset R(\alpha_-+V_-^w)$ and $K\subset L(\alpha_++V_+^w)$ (see Figure \ref{figlr}), where $\alpha_-$ and $\alpha_+$ are the transverse lines given by the Lemma \ref{lemacurvas} (note that we can find such lines since $K$ is compact, and therefore bounded).  Then we have that  $\phi(\bar s)\in R(\alpha_-+V_-^w)\cap L(\alpha_++V_+^w)$, and so, as remarked after Definition \ref{defequivalence}, $\phi(s)\in R(\alpha_-+V_-^w)$, for every $s>\bar s$, and $\phi(s)\in L(\alpha_++V_+^w)$, for every $s<\bar s$, that is, $[\phi]\subset R(\alpha_-+V_-^w)\cup L(\alpha_++V_+^w)$. However, since $\alpha_-$ is directed by $v_-$ such that $\langle v_-,\rho_0^\perp\rangle<0$ and $\langle v_-,\rho_0\rangle>0$ (and $\alpha_+$ is directed by $v_+$ such that $\langle v_+,\rho_0^\perp\rangle>0$ and $\langle v_+,\rho_0\rangle>0$) and, by the Lemma \ref{lemacone},  $[\g+w]$ is contained in $B(L_0,C_{\tilde{z}_0}^{\rho_0}(v)+w)\cup B(L_1,\tilde{z}_0+w)$, where the cone $C_{\tilde{z}_0}^{\rho_0}(v)$ is generated by $\rho_0$ and $v$ is a unit vector such that $\langle v,\rho_0\rangle>0$ and $0< \langle v,\rho_0^\perp\rangle<\langle v_+,\rho_0^\perp\rangle$, we have that there is $\bar t=\max\{t\in\R\mid(\g+w)(t)\in R(\alpha_-+V_-^w)\cup L(\alpha_++V_+^w)\}$, and therefore $\phi$ can only cross $\g+w$ before $\bar t$. Proceeding symmetrically we can obtain a lower bound for the instant in which $\phi$ can cross $\g+w$, thus proving the statement.

\begin{figure}[h!t]
\centering
\includegraphics[scale = 0.75]{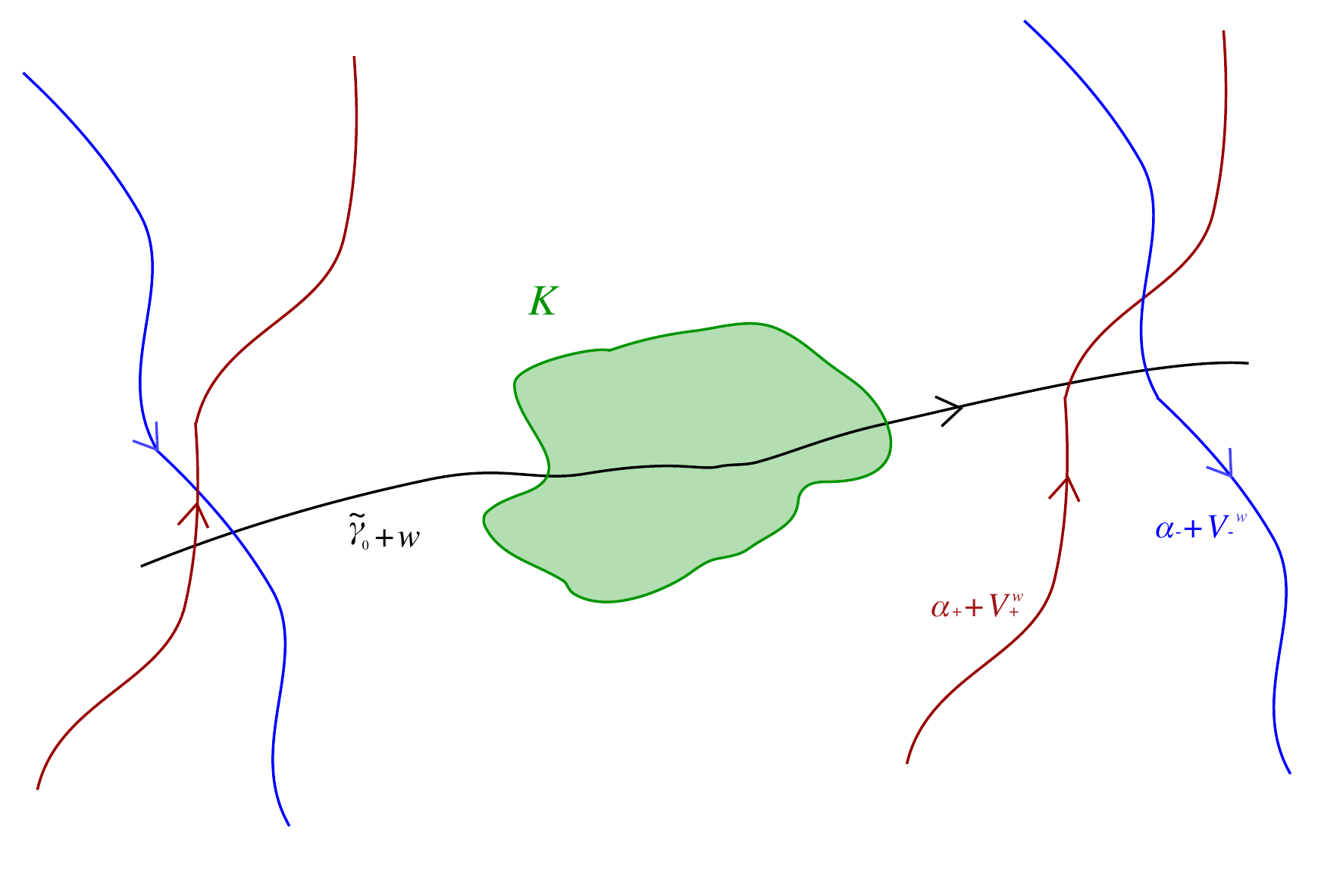}
\caption{Proof of Lemma \ref{lemacompacto}}
\label{figlr}
\end{figure}
\end{proof}

\begin{corolario}\label{lemadiresq}
Let $\zeta:[a,b]\to\R^2$ be a transverse path and $w\in\Z^2$. If $[\zeta]\cap l(\g+w)\neq\emptyset$ and $[\zeta]\cap r(\g+w)\neq\emptyset$, then $\zeta\pitchfork_{\tilde\F}(\g+w)$.
\end{corolario}
\begin{proof}
By Proposition \ref{proprl}, it is sufficient to prove that the set
$$I_\zeta=\{t\in\R\mid\phi_{(\g+w)(t)}\textrm{ crosses }\zeta\}$$
 is compact. But this follows directly from Lemma \ref{lemacompacto}.
\end{proof}

\begin{lema}\label{lemalinhanovo}
If $\tilde\gamma':\R\to\R^2$ is such that $\tilde\gamma'\sim_{\tilde\F}\g$, then $\tilde\gamma'$ is a line.
\end{lema}
\begin{proof}
By Lemma \ref{lemalinhaunica}, since $\tilde\gamma'\sim_{\tilde\F}\g$, we have that $\tilde\gamma'$  intersects each leaf at most once, therefore  $\tilde\gamma'$ is a simple path. Now let $K\subset\R^2$ be a compact set. Since $\tilde\gamma'\sim_{\tilde\F}\g$, we have that there is reparametrization of $\tilde\gamma'$ so that $\phi_{\g(t)}=\phi_{\tilde\gamma'(t)}$, for every $t\in\R$. Note that $\tilde\gamma'$ being proper is a property that does not depend on its parametrization. So, by Lemma \ref{lemacompacto}, we have that there is $M>0$ such that $\phi_{\tilde\gamma'(t)}\cap K=\emptyset$, for every $|t|>M$. Thus, if $t'\in(\tilde\gamma')^{-1}(K)$, we have that $|t'|<M$, and so we have that $\tilde\gamma'$ is a proper path, proving the lemma.
\end{proof}

\begin{lema}\label{lemagamaw}
Given $w\in\Z^2_*$, if $\g$ and $\g+w$  intersect the same leaf $\phi$, then there are sequences $t^+_k,s^+_k\nearrow+\infty$, $t^-_k,s^-_k\searrow-\infty$ and transverse paths $\tilde\gamma'_i:\R\to\R^2$, $i=1,2$ such that:
\begin{enumerate}[(i)]
\item $\tilde\gamma'_1\sim_{\tilde\F}\g$ and $\tilde\gamma'_2\sim_{\tilde\F}\g+w$;
\item $\tilde\gamma'_1(t^+_k)=\tilde\gamma'_2(s^+_k)$ and $\tilde\gamma'_1(t^-_k)=\tilde\gamma'_2(s^-_k)$, for every $k\in\N$.
\end{enumerate}
\end{lema}
\begin{proof}
Since $\g$ is recurrent, we can find sequences $t^+_k,s^+_k\nearrow+\infty$,  such that $\g(t_k^+)$ and $(\g+w)(s_k^+)$ belong to the leaf $\phi_k^+=\phi+w_k^+$, for some $w_k^+\in\Z^2$ (and in the same way, sequences $t^-_k,s^-_k\searrow-\infty$ with the same property). Let, for each $k\in \N$, $W_k^+$ be a tubular neighborhood of the leaf $\phi+w_k^+$ such that  $\g(t_k^+),(\g+w)(s_k^+)\in W_k^+$ and $W_k^-$ a tubular neighborhood of the leaf $\phi+w_k^-$ such that  $\g(t_k^-),(\g+w)(s_k^-)\in W_k^-$. Since $||w_k^\pm||\to\infty$, we can assume that the neighborhoods are mutually disjoint. So, making in each neighborhood a modification in the ways similar to the one made in the Lemma \ref{lemacurvas} (see Figure \ref{figlemacurvas}) we can obtain transverse paths $\tilde\gamma'_1$ and $\tilde\gamma'_2$ which satisfy the properties $(i)$ and $(ii)$ of the statement.
\end{proof}

\begin{obs}\label{obslinhas}
Note that by Lemma \ref{lemalinhanovo}, we have that $\tilde\gamma'_1$ and $\tilde\gamma'_2$ are lines. 
\end{obs}

\begin{lema}\label{lemacompconexa}
If $\g$ and $\g+w$ cross the same leaf, then every connected component of $L(\tilde\gamma'_1)\cap R(\tilde\gamma'_2)$ and of $R(\tilde\gamma'_1)\cap L(\tilde\gamma'_2)$ is bounded, where $\tilde\gamma'_1$ and $\tilde\gamma'_2$ are the paths given by Lemma \ref{lemagamaw}.  
\end{lema}
\begin{proof}
Let us assume, by contradiction, that the result is not true. We can, without loss of generality, assume that there is an unbounded connected component $O$ of $R(\tilde\gamma'_1)\cap L(\tilde\gamma'_2)$. The case where there is an unbounded connected component of $L(\tilde\gamma'_1)\cap R(\tilde\gamma'_2)$ is similar. Since $\tilde\gamma'_1$ and $\tilde\gamma'_2$ are lines that cross each other, any connected component of the complement of $[\tilde\gamma'_1]\cup[\tilde\gamma'_2]$ has on its boundary points that are in $[\tilde\gamma'_1]$ but not in $[\tilde\gamma'_2]$ and also point that are in $[\tilde\gamma'_2]$ but not in $[\tilde\gamma'_1]$. Let $P_1$ and $P_2$ be points in $\partial O\cap[\tilde\gamma'_1]\cap (\R^2\setminus[\tilde\gamma'_2])$ and $\partial O\cap(\R^2\setminus[\tilde\gamma'_1])\cap [\tilde\gamma'_2]$, respectively.  Let $\phi_1,\phi_2:\R\to \textrm{dom}(\tilde I)$ be the leaves of $\tilde\F$ passing through $P_1$ and $P_2$ respectively, and let us assume that $\phi_1(0)=P_1$ and $\phi_2(0)=P_2$. Note that $\phi_1(0)\in L(\tilde\gamma'_2)$, and since every leaf of $\tilde\F$ that intersects $\tilde\gamma'_1$ or $\tilde\gamma'_2$ must cross that path from left to right, we have that $\phi_1((-\infty, 0))$ is contained in $L(\tilde\gamma'_1)\cup L(\tilde\gamma'_2)$. Moreover, if $\phi_1((-\infty, 0))$ is bounded, then the $\alpha$-limit set of $\phi_1$ will be contained in $l(\tilde\gamma'_1)\cup l(\tilde\gamma'_2)$. Furthermore, if $\varepsilon>0$ is sufficiently small, then $\phi_1(\varepsilon)$ belongs to $O$. Analogously, it is possible to show that $\phi_2((0,+\infty))$ is contained in $R(\tilde\gamma'_1)\cup R(\tilde\gamma'_2)$,  its $\omega$-limit set is contained in  $r(\gamma_1)\cup r(\gamma_2)$ and, if $\varepsilon>0$ is sufficiently small, then $\phi_2(-\varepsilon)$ belongs to $O$.

We know that, since $\g$ and $\g+w$ have no $\tilde\F$-transverse intersection, the same holds for $\tilde\gamma'_1$ and $\tilde\gamma'_2$, since these paths are equivalent to $\g$ and $\g+w$, respectively. Therefore, by the Corollary \ref{lemadiresq}, we have that $\tilde\gamma'_1$ can not intersect both $r(\tilde\gamma'_2)$ and $l(\tilde\gamma'_2)$. Let us assume initially that $\tilde\gamma'_1$ does not intersect $l(\tilde\gamma'_2)$.

Let $t_0$ be such that $\tilde\gamma'_1(t_0)=P_1$. By the Lemma \ref{lemagamaw} there are $t_{-}<t_0<t_{+}$ and $s_{-}<s_{+}$ such that $\tilde\gamma'_1(t_{-})=\tilde\gamma'_2(s_{-})$ and $\tilde\gamma'_1(t_{+})=\tilde\gamma'_2(s_{+})$. Let $s'_{-}$ be the largest real number smaller than $s^{+}$ such that $\tilde\gamma'_2(s'_{-})= \tilde\gamma'_1(t'_{-})$, with $t'_{-}<t_0$, and let $s'_{+}$ be the smallest real number bigger than $s'_{-}$ such that $\tilde\gamma'_2(s'_{+})= \tilde\gamma'_1(t'_{+})$, with $t'_{+}>t_0$.  We have that, with the proper orientation, $\tilde\gamma'_1([t'_{-}, t'_{+}])\cup \tilde\gamma'_2([s'_{-}, s'_{+}])$ is the image of a simple closed curve $C_1$, separating the plane into two disjoint, one of them being bounded, $P_1$ belongs to the image of this curve, and if $\varepsilon$ is sufficiently small, then $\phi_1(-\varepsilon)$ and $\phi_1(\varepsilon)$ belong to distinct components of the complement of the curve. But we have that, if $H_1$ is a connected component of $l(\tilde\gamma'_2)$ which contains the $\alpha$-limit set of $\phi_1$, then $F_1=H_1\cup\phi_1((-\infty,-\varepsilon))$ does not intersect $[\tilde\gamma'_1]\cup[\tilde\gamma'_2]$. Furthermore, by the Lemma \ref{lemailimitada} we have that $H_1$ is an unbounded set, therefore $F_1$ is also unbounded, and thus is contained in an unbounded connected component of the complement of $C_1$. Note also that $\phi_1(-\varepsilon)$ is in the same connected complement component of $C_1$ as $F_1$. However, $O$  is contained in the complement of $[\tilde\gamma'_1]\cup[\tilde\gamma'_2]$, and therefore in the complement of $C_1$, and $O$ is also unbounded. Note that $\phi_1(\varepsilon)\in O$ and that $\phi_1(-\varepsilon)$ and $\phi_1(\varepsilon)$ are in separate components of the complement of $C_1$. Then $O$ and $F_1$ are contained in distinct components of the complement of $C_1$ and both are unbounded, which is absurd.

Let us now assume that $\tilde\gamma'_1$ does not intersect $ r(\tilde\gamma'_2)$. Let $s_0$ be such that $\tilde\gamma'_2(s_0)=P_2$. We can, as before, find $t''_{-}<t''_{+}$ and $s''_{-}<s_0<s''_{+}$ such that $\tilde\gamma'_1([t''_{-}, t''_{+}])\cup \tilde\gamma'_2([s''_{-}, s''_{+}])$ is the image of a simple closed curve $C_2$, and such that if  $\varepsilon$ is sufficiently small, $\phi_2(-\varepsilon)$ and $\phi_2(\varepsilon)$ belong to distinct components of the complement of the curve $C_2$. Note that, if $H_2$ is the connected component of $r(\tilde\gamma'_2)$ which contains the $\omega$-limit set of $\phi_2$, then $F_2=H_2\cup\phi_2((\varepsilon, +\infty)$ does not intersect $[\tilde\gamma'_1]\cup[\tilde\gamma'_2]$. Proceeding as before we show that the two connected components of the complement of $ C_2 $ are unbounded, thus obtaining a contradiction.
\end{proof}

\begin{lema}\label{lemalrnovo2}
If $r(\g)\cap l(\g+w)\neq\emptyset$, then:
\begin{enumerate}[(i)]
\item $[\g]\cap l(\g+w)\neq\emptyset$, $[\g]\cap r(\g+w)=\emptyset$;
\item $[\g+w]\cap r(\g)\neq\emptyset$, $[\g+w]\cap l(\g)=\emptyset$.
\end{enumerate}
\end{lema}
\begin{proof}
Note that if $\g$ and $\g +w$ do not cross a common leaf the result is trivial. Suppose then that $\g$ and $\g+w$  cross the same leaf, and let $\tilde\gamma'_1$ and $\tilde\gamma'_2$ be the lines given by Lemma \ref{lemagamaw}. Note that, since $\tilde\gamma'_1\sim_{\tilde\F}\g$ and $\tilde\gamma'_2\sim_{\tilde\F}\g+w$, we have $r(\g)=r(\tilde\gamma'_1)$ and $l(\g+w)=l(\tilde\gamma'_2)$. Let  $\tilde p'\in r(\tilde\gamma'_1)\cap l(\tilde\gamma'_2)$. Since $r(\tilde\gamma'_1)\subset R(\tilde\gamma'_1)$ and $l(\tilde\gamma'_2)\subset L(\tilde\gamma'_2)$, we have that $\tilde p'$ belongs to $R(\tilde\gamma'_1)\cap L(\tilde\gamma'_2)$ and so, by the Lemma \ref{lemacompconexa}, belongs to a bounded connected component of $R(\tilde\gamma'_1)\cap L(\tilde\gamma'_2)$. But the connected component $C$ of $r(\tilde\gamma'_1)$ which contains  $\tilde p'$ is unbounded, by Lemma \ref{lemailimitada}, and therefore needs to intersect $[\tilde\gamma'_1]\cup[\tilde\gamma'_2]$. Since $C$ is disjoint from $[\tilde\gamma'_1]$, we have that it intersects $[\tilde\gamma'_2]$,which implies that there is $\bar t$ such that $[\phi_{\tilde\gamma'_2(\bar t)}]\subset r(\g)$ and since $\tilde\gamma'_2\sim_{\tilde\F}\g+w$, this implies that some point of $\g + w$ is on the leaf $\phi_{\tilde\gamma'_2(\bar t)}$, and therefore is also in $r(\g)$. Since $\g$ and $\g+w$ do not have a transverse intersection, we deduce by Corollary~\ref{lemadiresq} that $[\g+w]\cap l(\g)=\emptyset$. The proof that $l(\g+w)$ intersects $\g$ and therefore $[\g]\cap r(\g+w)=\emptyset$  is analogous.
\end{proof}

\begin{obs}\label{obslrnovo}
Note that we can obtain a result symmetrical to the previous one, with  $r(\g+w)\cap l(\g)\neq\emptyset$. 
\end{obs}

Let $p$ be a singularity of $\F$ and let us fix $\tilde{p}$ a lift of $p$. We have that $\tilde{p}$ is a singularity of $\tilde\F$ and suppose that $\tilde{p}\in l(\g)$.

\begin{lema}\label{lemacorte}
Let $w,w'\in\Z^2$ be such that  $\langle w,\rho_0^\perp\rangle>0$ and $\tilde{p}+w'\in l(\g)$, then $(\tilde{p}+w')+w\in l(\g)$.
\end{lema}
\begin{proof}
Suppose, by contradiction, that $(\tilde{p}+w')+w\in r(\g)$. Note that  $(\tilde{p}+w')+w\in l(\g+w)$, so we have that $l(\g+w)\cap r(\g)\neq\emptyset$ (see Figure \ref{fig46}). Therefore, by Lemma \ref{lemalrnovo2}, we have that $[\g+w]\cap r(\g)\neq\emptyset$.

\begin{figure} 
\centering
\includegraphics[scale = 0.75]{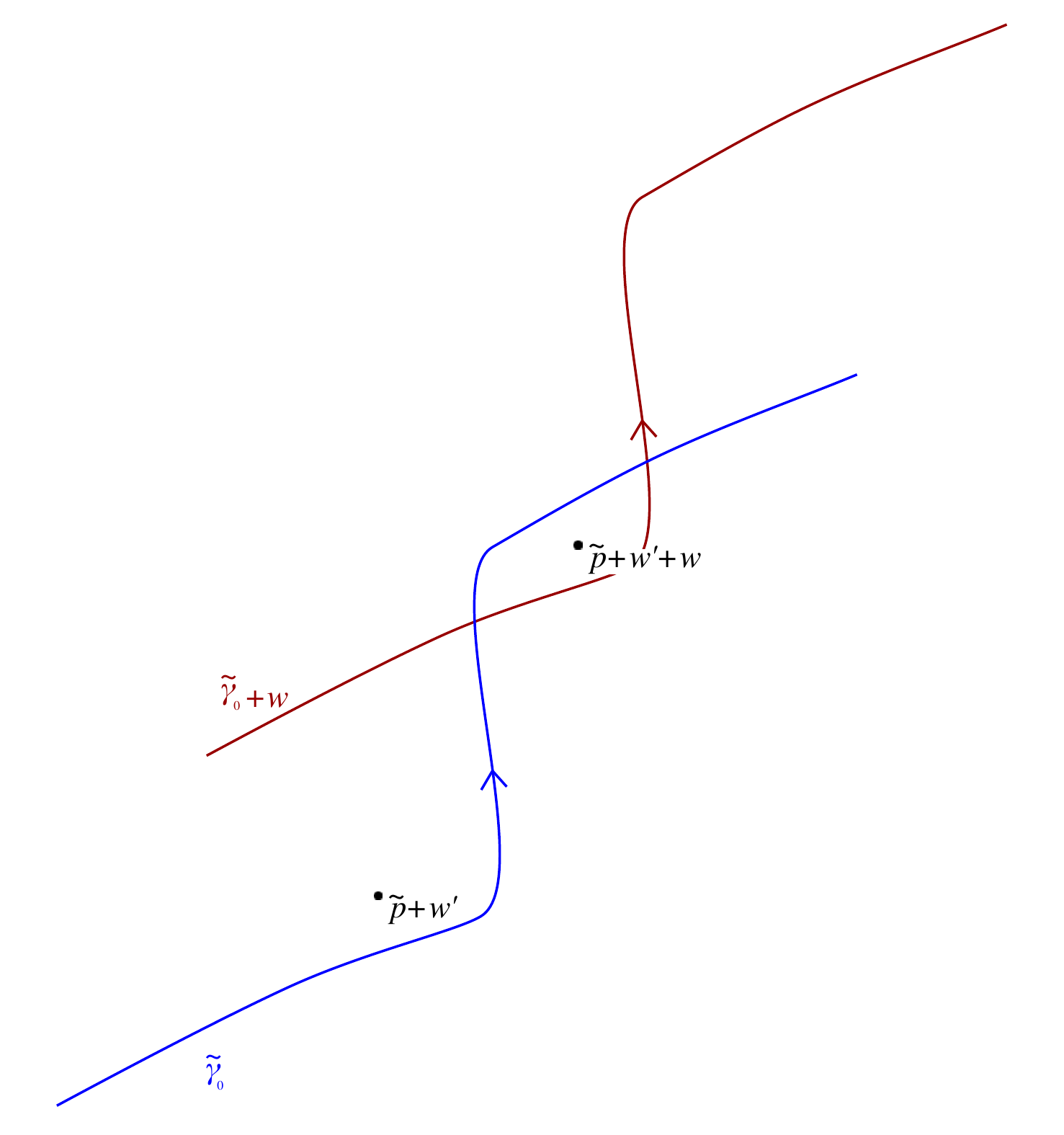}
\caption{Illustration of Lemma \ref{lemacorte}
\label{fig46}}
\end{figure}

We claim that this implies that $r(\g+w)\subset \R^2\setminus l(\g)$: indeed, if $r(\g+w)\cap l(\g)\neq\emptyset$, we have again by the Lemma \ref{lemalrnovo2} that $[\g+w]\cap l(\g)\neq\emptyset$. So, we get by Proposition \ref{proprl} that $\g\pitchfork_{\F}(\g+w)$, which is a contradiction by the Lemma \ref{lemaautogamma}, proving the claim. Therefore, since $(\tilde{p}+w')+2w$ is a singularity and so it must be contained in $r(\g)\cup l(\g)$, and since $(\tilde{p}+w')+2w\in r(\g+w)$ and $r(\g+w)\subset \R^2\setminus l(\g)$, we have that $(\tilde{p}+w')+2w\in r(\g).$ So,  by induction, we have that $(\tilde{p}+w')+nw\in r(\g)$, for all $n\in\N$.

Let us denote now $v'_w=w/||w||$. If $\theta$ is the smallest angle between $v'_w$ and the  line generated by $\rho_0$, let $v_w$ be a unit vector such that its angle to the line generated by $\rho_0$ is $\theta/2$. Thus, we have that for $n$ sufficiently large, $(\tilde{p}+w')+nw$ is in the connected component of $\R^2\setminus (B(L_0,C_{\tilde{z}_0}^{\rho_0}(v_w))\cup B(L_1,\tilde{z}_0))$ contained in the left of the line generated by $\rho_0$ passing through $\tilde{z}_0$ (see Figure \ref{figlemacorte}). But, by Lemma \ref{lemacone}, we have that $[\g]\subset B(L_0,C_{\tilde{z}_0}^{\rho_0}(v_w))\cup B(L_1,\tilde{z}_0)$, therefore $(\tilde{p}+w')+nw\in l(\g)$, which is a contradiction.

\begin{figure} 
\centering
\includegraphics[scale = 0.5]{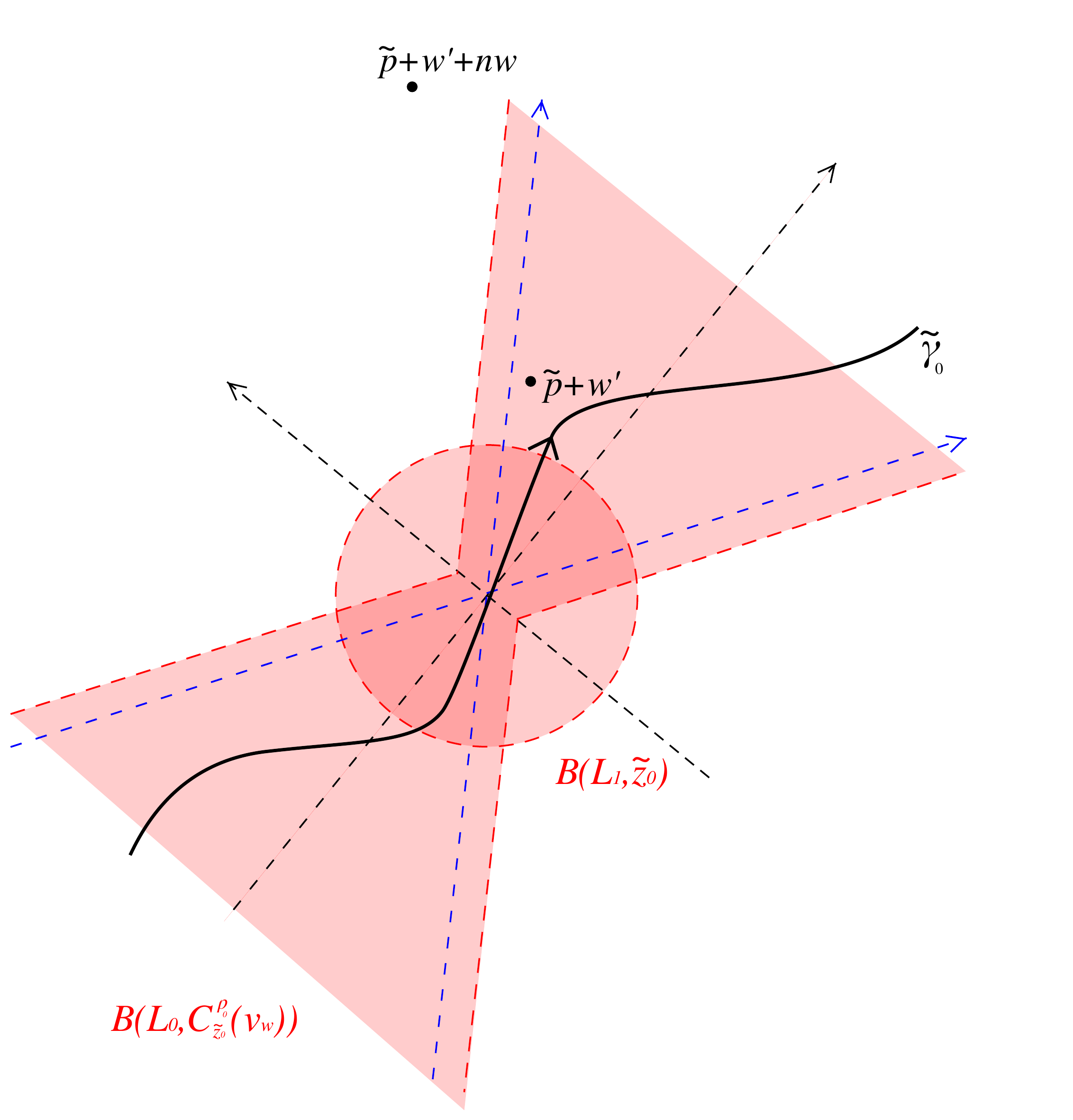}
\caption{Illustration of Lemma \ref{lemacorte}
\label{figlemacorte}}
\end{figure}
\end{proof}

Note that with an analogous demonstration we can obtain the following:

\begin{lema}\label{lemacorte2}
Let $w,w'\in\Z^2$ be such that $\langle w,\rho_0^\perp\rangle<0$ and $\tilde{p}+w'\in r(\g)$, then $(\tilde{p}+w')+w\in r(\g)$.
\end{lema}

Since the slope of $\rho_0^\perp$ is irrational, we can define the following order in $\Z^2$:

\begin{definicao}
$w\succ w^\prime \Leftrightarrow\langle w-w^\prime,\rho_0^\perp\rangle>0$.
\end{definicao} 

\begin{obs}
Note that such order is defined simply by projecting $\Z^2$ on the line generated by $\rho_0^\perp $ and using the natural order of such a line. In addition, we have that such projection is dense on the line.
\end{obs}

We will denote by $Z_r=\{w\in\Z^2 \mid \tilde{p}+w\in r(\g)\}$ and $Z_l=\{w\in\Z^2\mid \tilde{p}+w\in l(\g)\}$. Note that, since $(\tilde p+\Z^2)\cap[\g]=\emptyset$, we have that $\Z^2=Z_l\cup Z_r$. The following lemma will show that the projections of $Z_r$ and $Z_l$ on the line generated by $\rho_0^\perp $ are contained in two disjoint half-lines.

\begin{lema}\label{lemasemiretas}
The sets $Z_r$ and $Z_l$ defined above satisfy the following properties:
\begin{enumerate}[(i)]
\item If $w\in Z_l$ and $w'\succ w$, then $w'\in Z_l$;
\item If $w\in Z_r$ and $w'\prec w$, then $w'\in Z_r$.
\end{enumerate}
\end{lema}
\begin{proof}
To prove $(i)$, let us take $w\in Z_l$ and $w'\succ w$. So we have $\langle w'-w,\rho_0^\perp\rangle>0$. Since $w\in Z_l$, we have by definition that $\tilde{p}+w\in l(\g)$ and then, by Lemma \ref{lemacorte}, we have that $\tilde{p}+w+(w'-w)\in l(\g)$, and so $w'\in Z_l$.

The proof of $(ii)$ is analogous.
\end{proof}

\begin{lema}\label{lemacorte3}
Let $w\in\Z^2$ be such that $w\prec 0$. So there is $w^*\in\Z^2$ such that $w^*\in Z_l$ and $w^*+w\in Z_r$.
\end{lema}
\begin{proof}
Let us denote $\langle A,\rho_0^\perp\rangle=\{\langle w,\rho_0^\perp\rangle\mid w\in\ A\}$. Using such notation, we have that $\langle \Z^2,\rho_0^\perp\rangle=\langle Z_l,\rho_0^\perp\rangle\cup\langle Z_r,\rho_0^\perp\rangle$ and $\langle \Z^2,\rho_0^\perp\rangle$ is dense in $\R$. Let us prove that $\langle Z_l,\rho_0^\perp\rangle$ is bounded from below. By contradiction, if $\langle Z_l,\rho_0^\perp\rangle$ is unbounded from below, there are $w_n\in\Z^2, n\in\N$, such that $\langle w_n,\rho_0^\perp\rangle\to-\infty$. But, by Lemma \ref{lemasemiretas}, if $w_n\in Z_l$, then $w'\in Z_l$, if $w'\succ w_n$, and so we have that $\langle Z_l,\rho_0^\perp\rangle=\langle \Z^2,\rho_0^\perp\rangle$, which is a contradiction. Therefore $\langle Z_l,\rho_0^\perp\rangle$ is bounded from below. Analogously, we have that $\langle Z_r,\rho_0^\perp\rangle$ is bounded from above. Let us denote $l^*=\inf\langle Z_l,\rho_0^\perp\rangle$ and $r^*=\sup\langle Z_r,\rho_0^\perp\rangle$. We claim that $l^*=r^*$. In fact, if $r^*<l^*$, we have a contradiction, because we would have $(r^*,l^*)\cap\langle \Z^2,\rho_0^\perp\rangle=\emptyset$, and $\langle \Z^2,\rho_0^\perp\rangle$ is dense in $\R$. If $l^*<r^*$, we have that there are $w_l\in Z_l$ and $w_r\in Z_r$ such that $w_r-w_l\succ 0$. Thus, by the Lemma \ref{lemacorte}, we have that $w_l+(w_r-w_l)\in l(\g)$, which is a contradiction. Therefore, $l^*=r^*$. Since $l^*=\inf \langle Z_l,\rho_0^\perp\rangle$, we have that there is a sequence $w_n\in Z_l$ such that $\langle w_n,\rho_0^\perp\rangle\to l^*$. So, since $\langle w,\rho_0^\perp\rangle<0$, there is $n'\in \N$ such that $\langle w_{n'}+w,\rho_0^\perp\rangle<l^*$. Thus, making $w^*=w_{n'}$, we have that $w^*\in Z_l$ and $w^*+w\in Z_r$, proving the lemma.
\end{proof}

\begin{obs}\label{obscorte}
Analogously to the previous Lemma, given $w\succ 0$, we can obtain $w^{**}\in\Z^2$ such that $w^{**}\in Z_r$ and $w^{**}+w\in Z_l$
\end{obs}

\begin{lema}\label{lema1}
Let $w\in\Z^2_*$. If $w\succ 0$, then for every $M>0$ there exists $t_M^{+}>M$ and $t_M^{-}<-M$ such that both $\g(t_M^{-})$ and $\g(t_M^{+})$ lie in $r(\g+w)$. Likewise, there exists $s_M^{+}>M$ and $s_M^{-}<-M$ such that both $\g+w(s_M^{-})$ and $\g+w(s_M^{+})$ lie in $l(\g)$. In particular we have that $\g$ and $\g+w$ are not $\tilde\F$-equivalent.
\end{lema}
\begin{proof}
Let's look at the case where $w\prec 0$. By the Lemma \ref{lemacorte3}, we have that there is $w^*\in\Z^2$ such that $\tilde p+w^*\in l(\g)$ and $\tilde p+w^*+w\in r(\g)$, and therefore $\tilde p+w^*+w\in r(\g)\cap l(\g+w)$. By Lemma \ref{lemalrnovo2}, we have that $[\g]\cap l(\g+w)\neq\emptyset$ and $[\g+w]\cap r(\g)\neq\emptyset$, so one can find some $t_0, s_0$ such that $\g(t_0)\in l(\g+w)$ and $\g+w(s_0)\in r(\g)$. Fix $M>0$ and let us show the existence of $t_M^{+}$, the other cases are similar.  Let $\phi_0$ be the leaf that passes through $\g(t_0)$ and note that, if $w^*\succ 0$ is in $\Z^2$, then $\phi_0\subset l(\g+w+w^*)\subset l(\g+w)$. Let $N>|t_0$ be an integer. Using Proposition \ref{proppqeps}, one can find a sequence $(p_l, q_l)$ in $\Z^2\times\N$ with $q_l$ going to infinity and $p_l\succ 0$ such that $\tilde f^{q_l}(\tilde z_0)-p_l$ converges to $\tilde z_0$. This implies, by Lemma \ref{lemaestabilidade}, that for sufficiently large $l$ the path $\g-p_l\mid_{[q_l-N-1,q_l+ N+1]}$ contains a subpath equivalent to $\g\mid_{[-N,N]}$. But since  $\g\mid_{[-N,N]}$intersects $\phi_0$, $\g-p_l\mid_{[q_l-N-1,q_l+ N+1]}$ must also do so, and therefore $\g\mid_{[q_l-N-1,q_l+ N+1]}$ intersects $\phi_0+p_l$, which lies in $l(\g+w)$. It suffices then to take $l$ such that $q_l>M-N-1$. 

The proof for $w\succ 0$ is analogous, using the Remark \ref{obscorte}.
\end{proof}

\begin{lema}\label{lema2}
Given $t_1>0$, there is $0<\varepsilon<\frac{1}{2}$ such that, if $\tilde z'\in B(\varepsilon,\tilde z_0)$, then every element of $\tilde{I}^{\Z}_{\tilde{\F}}(\tilde z')$ contains a sub-path $\tilde\F$-equivalent to $\g\vert_{[0, t_1]}$. 
\end{lema}
\begin{proof}
It follows directly by Lemma \ref{lemaestabilidade}.
\end{proof}

In what follows, we will say that $f$ \emph{does not have bounded deviation in the positive direction of $\rho_0^\perp$} if there are $\tilde x_k \in\R^2$ and $(n_k)_{k\in\N}$ an increasing sequence such that $\lim_{k\to\infty}\langle \tilde{f}^{n_k}(\tilde x_k)-\tilde x_k,\rho_0^\perp\rangle=+\infty$. Analogously, we will say that $f$ \emph{does not have bounded deviation in the negative direction of $\rho_0^\perp$} if there are $\tilde x_k$  and $(n_k)_{k\in\N}$ as before such that the previous limit is equal to $-\infty$. Note that if $f$ does not have bounded deviation in the direction of $\rho_0^\perp$ so either $f$ does not have bounded deviation in the positive direction of $\rho_0^\perp$ or $f$ does not have bounded deviation in the negative direction of $\rho_0^\perp$.

\begin{lema}\label{lema3}
If $f$ does not have bounded deviation in the positive direction of $\rho_0^{\perp}$, then, given $0<\varepsilon<\frac{1}{2}$, there are $\tilde x\in\R^2$, $N\in\N$ and $P\in\Z^2$ such that:
\begin{enumerate}[(i)]
\item $\tilde x\in B(\varepsilon, \tilde z_0)$
\item $P \succ (-2,0)$
\item $\tilde f^{N}(\tilde x)\in B(\varepsilon, \tilde z_0+P)$
\end{enumerate}
\end{lema}
\begin{proof}
We have by the Lemma \ref{lemaessencial} that the set $U_{\varepsilon}=\bigcup_{i=0}^{\infty} f^{i}(\pi(B(\varepsilon, \tilde{z}_0)))$ is fully essential, and since $\varepsilon<\frac{1}{2}$, we have that $\overline{\pi(B(\varepsilon,\tilde{z}_0))}$ is inessential. So, applying the Proposition \ref{propguelman2015rotation}, we get a compact set of the plane $K$ and $M\in\N$ such that $[0,1]^2$ is contained in a bounded connected component of $\R^2\setminus K$, which we shall denote by $A$, and $K\subset\bigcup_{|i|\leq M,\, ||v||_{\infty}\leq M}\left(\tilde{f}^i(B(\varepsilon, \tilde{z}_0))+v\right)$. Since $f$ does not have bounded deviation in the positive direction of $\rho_0^{\perp}$, there are $P'\in\Z^2$ and $l\in\N$ such that $l>2M$, $\langle P', \rho_0^{\perp}\rangle > \langle -(2,0), \rho_0^{\perp}\rangle+2M$ and $\tilde{f}^{l}([0,1]^2)\cap([0,1]^2+P')\neq\emptyset$. Since $[0,1]^2\subset A$ and $\tilde{f}^{l}([0,1]^2)\cap([0,1]^2+P')\neq\emptyset$, we have that $\tilde{f}^{l}(A)$ intersects $A+P'$. Since $A$ is bounded, we have that $\tilde{f}^{l}(\partial A)\cap (\partial A+P')\neq\emptyset$, and since $\partial A\subset K$, we get $\tilde{f}^{l}(K)\cap(K+P')\neq\emptyset$. Now, let $\tilde y\in\tilde{f}^{l}(K)\cap(K+P')$. So there are $n_i\in\Z$, $|n_i|<M$ and $v_i\in\Z^2$, $||v_i||_\infty <M$, for $i=1,2$, such that 
\begin{align*}
\tilde y\in \tilde{f}^l(K) &\Rightarrow \tilde y \in \tilde{f}^l(\tilde{f}^{n_1}(B(\varepsilon,\tilde z_0 ))+v_1)=\tilde{f}^{l+n_1}(B(\varepsilon,\tilde z_0 ))+v_1 \\
\tilde y \in K+P' &\Rightarrow \tilde y \in \tilde{f}^{n_2}(B(\varepsilon,\tilde z_0))+v_2+P'.
\end{align*} 
Then we get
\begin{align*}
\tilde{f}^{-n_2}(\tilde y) &\in(\tilde{f}^{l+n_1-n_2}(B(\varepsilon,\tilde z_0))+v_1)\cap B(\varepsilon,\tilde z_0+v_2+P')\\
\tilde{f}^{-n_2}(\tilde y) -v_1&\in\tilde{f}^{l+n_1-n_2}(B(\varepsilon,\tilde z_0))\cap B(\varepsilon,\tilde z_0+v_2-v_1+P').
\end{align*}
Thus, setting $N=l+n_1-n_2$ e $P=v_2-v_1+P'$, we get the result.
\end{proof}

Note that we can prove an analogous result for the case where $f$ does not have bounded deviation in the negative direction of $\rho_0^\perp$.

With all the results proven so far we can complete the proof of Theorem A.

\begin{proof}[Proof of  Theorem A]
Suppose by contradiction that $f$ does not have bounded deviation in the  direction of $\rho_0^\perp$. Let us assume that $f$ does not have bounded deviation in the positive direction of $\rho_0^\perp$ (the other case is analogous).

Since $(1,0)\prec 0$, we have by the Lemma \ref{lema1} that $\g$ and $\g+(1,0)$ are not $\tilde\F$-equivalents, and beyond that, $\g$ intersects a leaf in $l(\g+(1,0))$. Let $t_0$ be a moment in which such intersection occurs, i.e., $\g(t_0)$ belongs to a leaf, which we shall denote by $\phi_l$, which is contained in $l(\g+(1,0))$. Analogously, since $-(1,0)\succ 0$ we have by the Lemma \ref{lema1} that $\g$ and $\g-(1,0)$ are not $\tilde\F$-equivalents, and beyond that $\g$ intersects a leaf in $r(\g-(1,0))$. Let $t_1$ be such that  $\g(t_1)$ belongs to a leaf, which we shall denote by $\phi_r$, which is contained in $r(\g-(1,0))$. Note that, by Lemma \ref{lema1},we can assume both $t_0$ and $t_1$ are positive, and that $0<t_0<t_1$. Now, let $0<\varepsilon<\frac{1}{2}$ be given by Lemma \ref{lema2}, and also $\tilde x \in \R^2$, $N\in\N$ and $P\in\Z^2$ given by the Lemma \ref{lema3}. We will denote a fixed element of $\tilde{I}^{\Z}_{\tilde{\mathcal{F}}}(\tilde x)$ by $\bx$. 

\begin{figure} 
\centering
\includegraphics[scale = 0.75]{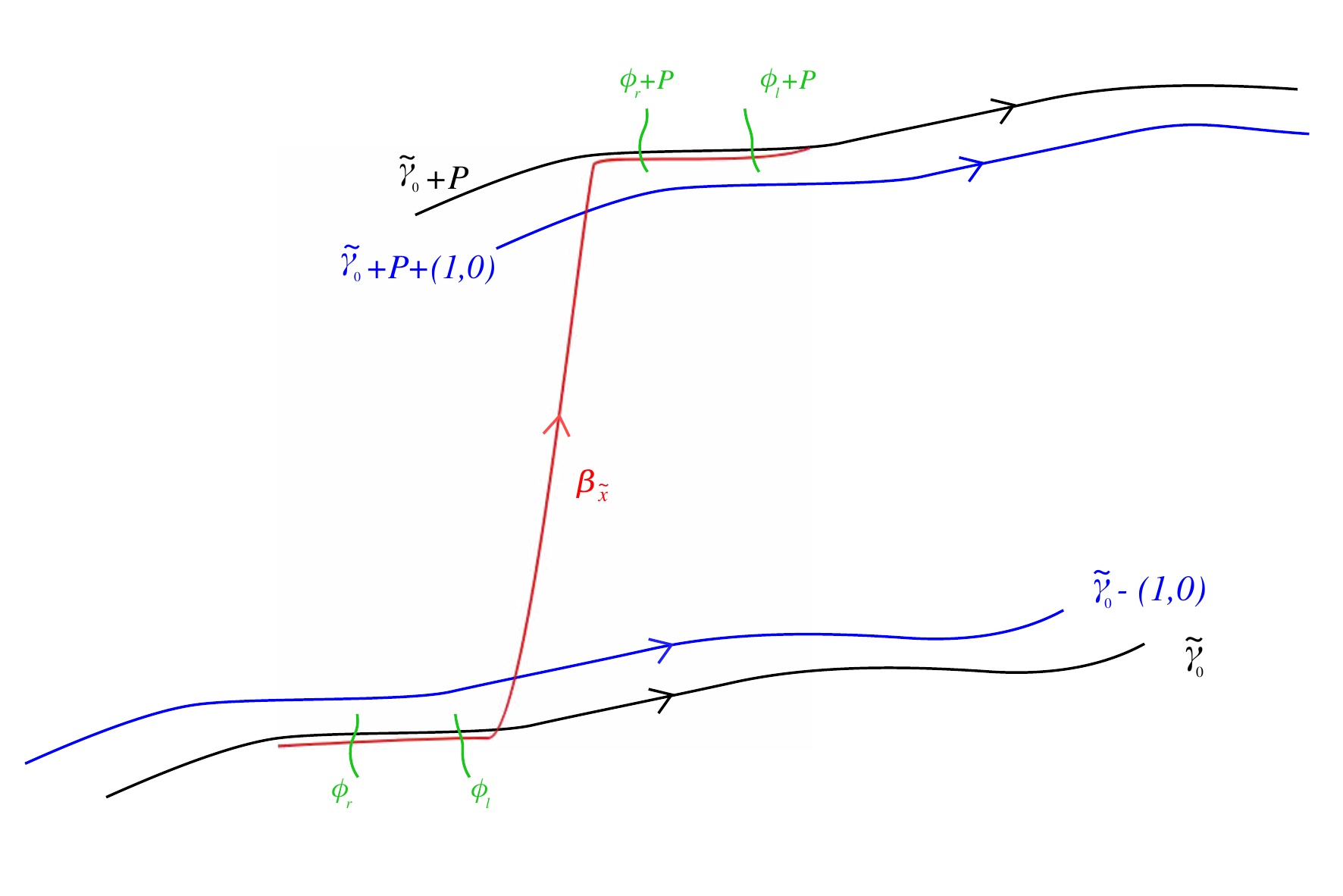}
\caption{Construction of $\bx$}
\label{fig1}
\end{figure}

Let us prove that $\bx$ intersects $r(\g-(1,0))$ and $l(\g+P+(1,0))$. Since $\tilde x\in B(\varepsilon,\tilde z_0)$, we have by the Lemma \ref{lema2} that $\g\vert_{[0, t_1]}$ is equivalent to a subpath of $\bx$, thus we have that $\bx$ crosses $\phi_r$. Similarly, since $\tilde f^N(\tilde x)\in\ B(\varepsilon,\tilde z_0+P)$, we have that $\g\vert_{[0,t_1]}+P$ is equivalent to a sub-path of $\bx$, and so $\bx$ intersects $\phi_l+P$. Let us denote by $I=[a,b]$ an interval such that  $[\bx(a)$ belongs to $\phi_r$, $\bx(b)$belongs to $\phi_l+P$.  

We claim that, for every $w\in\Z^2$ such that $-(1,0)\prec w\prec P+(1,0)$,   $(\g+w)$ intersects $\bx|_I$ $\tilde\F$-transversally.  Let us first prove that $[\phi_l+P]\subset l(\g+w)$ and $[\phi_r]\subset r(\g+w)$. Let us suppose by contradiction that $[\phi_r]\not\subset r(\g+w)$. We have two possibilities: $[\phi_r]\subset l(\g+w)$ or $\g+w$ crosses $\phi_r$. If $[\phi_r]\subset l(\g+w)$, we have that $[\g]\cap l(\g+w)\neq\emptyset$. In addition, since $0\prec-(1,0)\prec w$, we have by Lemma \ref{lema1} that $[\g]\cap r(\g+w)\neq\emptyset$. Therefore, by the Lemma \ref{lemadiresq}, we have that $\g\pitchfork_{\tilde\F}\g+w$, which is a contradiction, by the Lemma \ref{lemaauto}. If $\g+w$ crosses $\phi_r$, since $[\phi_r]\subset r(\g-(1,0))$, we have $[\g+w]\cap r(\g-(1,0))\neq\emptyset$. In addition, since $-(1,0)\prec w$, we have by Lemma \ref{lema1} that $[\g]\cap l(\g-(1,0)-w)\neq\emptyset$, and thus $[\g+w]\cap l(\g-(1,0))\neq\emptyset$. Again, by the Lemma \ref{lemadiresq} we have that $(\g-(1,0))\pitchfork_{\tilde\F}(\g+w)$, which is a contradiction, by the Lemma \ref{lemaauto}. Therefore we have that $[\phi_r]\subset r(\g+w)$. In a symmetrical way, using the fact that $w\prec P+(1,0)\prec P$, we can prove that $[\phi_l+P]\subset l(\g+w)$. So, by the Lemma \ref{lemadiresq}, we have that $(\g+w)\pitchfork_{\tilde{\F}}\bx|_I$, proving the claim. Note that, by Lemma \ref{lemasubcaminhosadmissiveis}, that both $\g+w$ and $\bx\mid_{I}$ are admissible.

Now let $w_1,w_2\in\Z^2$ be such that $-(1,0)\prec w_1\prec w_2\prec P+(1,0)$. Since $(\g+w_1)\pitchfork_{\tilde{\F}}\bx|_I$, we have that there are $t',s'\in\R$ such that $\g+w_1$ and $\bx$ intersect $\tilde\F$-transversally at $(\g+w_1)(t')=\bx(s')$. In particular, one can find an interval $J_1=[a_1, b_1]$ containing $t'$ such that $(\g+w_1)\mid_{J}\pitchfork_{\tilde\F}\bx$ and Lemma \ref{lemasubcaminhosadmissiveis} and Proposition \ref{propadm} show that, for any $c<a_1$ the path $beta_c'=(\g+w_1)|_{[c,t']}\,\bx|_{[s',b]}$ is admissible.
 Note that, by Lemma \ref{lema1}, there exists some $c_0<a_1$ such that $\phi_0=\phi_{(\g+w_1)(c_0)}$ is in $r(\g+w_2)$. This implies that, for any $c\leq c_0$, $\beta'_c \pitchfork_{\tilde\F} \g+w_2$, as it intersects both $\phi_0\subset r(\g+w_2)$ and $\phi_l+P\subset l(\g+w_2)$. Let $J_2=[a_2, b_2]$ be an interval such that $\beta'_c \pitchfork_{\tilde\F} (\g+w_2)\mid_{J_2}$. We can, as in the proof of Lemma \ref{lema1}, find $w_3$ such that there exists some interval $J_3=[a_3,b_3]$ with $b_3< c_0$ such that $(\g+w_3)\mid_{J_3}$ is equivalent to $(\g+w_2)\mid_{J_2}$.  But this implies that the path $\beta_{a_3}'$ has a transverse intersection with $\beta_{a_3}'+w_3-w_1$ since the latter has a subpath equivalent to $(\g+w_1)\mid_{j_3}+(w_3-w_1)=(\g+w_3)\mid_{j_3}$, a contradiction with Lemma  \ref{lemaauto}, concluding the demonstration.

\begin{figure} 
\centering
\includegraphics[scale = 0.75]{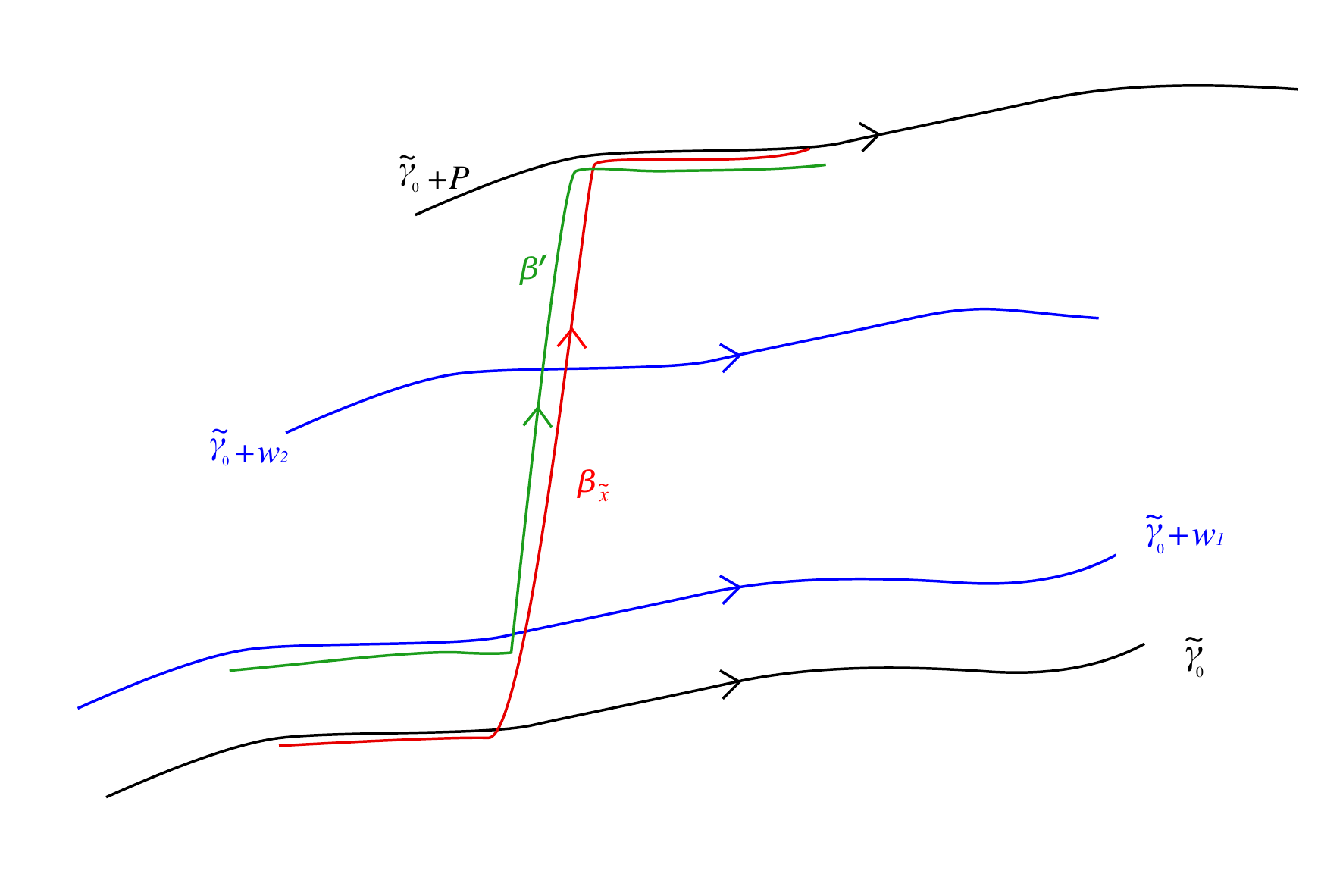}
\caption{Construction of $\beta'$}
\label{fig2}
\end{figure}

\begin{figure} 
\centering
\includegraphics[scale = 0.75]{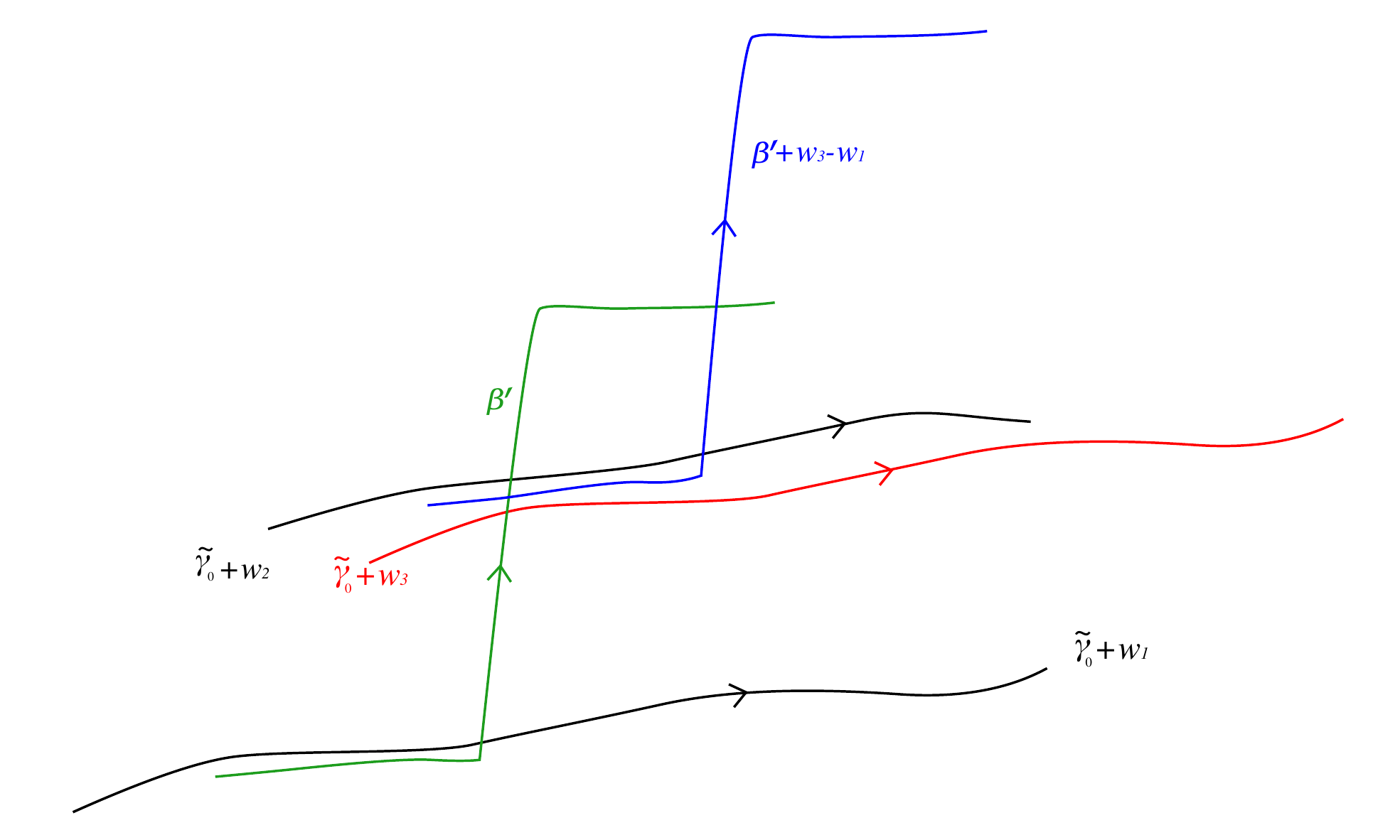}
\caption{Intersection of $\beta'$ with $\beta'+w_3-w_1$}
\label{fig3}
\end{figure}
\end{proof}

\subsection{Proof of Corollary B}
Assume $f:\T\to\T$ is isotopic to the identity, has a periodic point,  let $\tilde f:\R^2\to\R^2$ be a lift of $f$ and assume $\rho(\tilde f)$ is a non-degenerate line segment. Note that, as $f$ has a periodic point, $\rho(\tilde f)$ has at least one point in $\Q^2$. If $\rho(\tilde f)$ has two distinct points in $\Q^2$, then the result follows directly from \cite{Davalos2013SublinearDiffusion}. So we may assume that $\rho(\tilde f)$ has a single point with rational coordinates, and then Theorem C from \cite{LeCalvezTal2015ForcingTheory} implies that there exists integers $p_1, p_2, q$ and some vector $rho_0$  such that $\rho(\tilde f)\{(p_1/q, p_2/q) +t \rho_0 \mid 0\le t \le q\}$. But then we can take $g=f^q$ and its lift $\tilde g= \tilde f^q-(p_1, p_2)$ and apply Theorem A to them, deducing that $g$ has uniformly bounded deviations in the direction $\rho_0^{\perp}$. But if a power of $f$ has uniformly bounded deviations in a given direction, $f$ itself must also have this property, which concludes the proof of the Corollary.

\bibliography{dynamics} \bibliographystyle{alpha}

\begin{thebibliography}{GKT15}

\bibitem[AZ15]{AddasZanata2015BoundedMeanMotionDiffeos}
S.~Addas-Zanata.
\newblock Uniform bounds for diffeomorphisms of the torus and a conjecture of
  boyland.
\newblock {\em Journal of the London Mathematical Society}, 91(2):537--553,
  2015.

\bibitem[AZL19]{xioachuansalvador2019}
S.~Addas-Zanata and X.~Liu.
\newblock On stable and unstable behaviours of certain rotation segments.
\newblock {\em arXiv preprint arXiv:1903.08703}, 2019.

\bibitem[BCR20]{beguin2016fixed}
F.~B{\'e}guin, S.~Crovisier, and F.~Le Roux.
\newblock Fixed point sets of isotopies on surfaces.
\newblock {\em to appear n Journal of the European Mathematical Society}, 2020.

\bibitem[Bro12]{brouwer1912beweis}
L.~E.~J. Brouwer.
\newblock Beweis des ebenen translationssatzes.
\newblock {\em Mathematische Annalen}, 72(1):37--54, 1912.

\bibitem[Cal05]{le2005version}
P.~Le Calvez.
\newblock Une version feuillet{\'e}e {\'e}quivariante du th{\'e}oreme de
  translation de brouwer.
\newblock {\em Publications Math{\'e}matiques de l'Institut des Hautes
  {\'E}tudes Scientifiques}, 102(1):1--98, 2005.

\bibitem[CT15]{LeCalvezTal2015ForcingTheory}
P.~Le Calvez and F.~A. Tal.
\newblock Forcing theory for transverse trajectories of surface homeomorphisms.
\newblock {\em Inventiones mathematicae}, 212:619--729, 2015.

\bibitem[CT18]{calvez2018topological}
P.~Le Calvez and F.~A. Tal.
\newblock Topological horseshoes for surface homeomorphisms.
\newblock {\em arXiv preprint arXiv:1803.04557}, 2018.

\bibitem[CT19]{ConejerosTal2}
J.~Conejeros and F.~A. Tal.
\newblock Applications of forcing theory to homeomorphisms of the closed
  annulus.
\newblock {\em Preprint, {\tt arXiv:1909.09881}}, 2019.

\bibitem[Dav16]{Davalos2013SublinearDiffusion}
P.~Davalos.
\newblock On annular maps of the torus and sublinear diffusion.
\newblock {\em Inst. Math. Jussieu}, pages 1--66, 2016.

\bibitem[Fra88]{franks:1988a}
J.~Franks.
\newblock Recurrence and fixed points of surface homeomorphisms.
\newblock {\em Ergodic Theory Dynam.\ Systems}, 8*:99--107, 1988.

\bibitem[Fra89]{franks:1989}
J.~Franks.
\newblock Realizing rotation vectors for torus homeomorphisms.
\newblock {\em Trans. Amer. Math. Soc.}, 311(1):107--115, 1989.

\bibitem[GKT14]{GuelmanKoropeckiTal2012Annularity}
N.~Guelman, A.~Koropecki, and F.~Tal.
\newblock A characterization of annularity for area-preserving toral
  homeomorphisms.
\newblock {\em Math. Z.}, 276(3-4):673--689, 2014.

\bibitem[GKT15]{guelman2015rotation}
N.~Guelman, A.~Koropecki, and F.~A. Tal.
\newblock Rotation sets with non-empty interior and transitivity in the
  universal covering.
\newblock {\em Ergodic Theory and Dynamical Systems}, 35(3):883--894, 2015.

\bibitem[J{\"a}g09]{Jaeger2009Linearisation}
T.~J{\"a}ger.
\newblock Linearisation of conservative toral homeomorphisms.
\newblock {\em Invent. Math.}, 176(3):601--616, 2009.

\bibitem[JP15]{JaegerPasseggi2015SemiconjugateToIrrational}
T.~J{\"a}ger and A.~Passeggi.
\newblock On torus homeomorphisms semiconjugate to irrational circle rotations.
\newblock {\em Ergodic Theory Dynam.\ Systems}, 7(35):2114--2137, 2015.

\bibitem[JT16]{JaegerTal2016IrrationalRotationFactors}
T.~J{\"a}ger and F.~Tal.
\newblock Irrational rotation factors for conservative torus homeomorphisms.
\newblock {\em Ergodic Theory Dynam.\ Systems}, pages 1--10, 2016.

\bibitem[KK08]{kocsard/koropecki:2007}
A.~Kocsard and A.~Koropecki.
\newblock Free curves and periodic points for torus homeomorphisms.
\newblock {\em Ergodic Theory Dynam.\ Systems}, 28(06):1895--1915, 2008.

\bibitem[Koc16]{kocsard2016dynamics}
A.~Kocsard.
\newblock On the dynamics of minimal homeomorphisms of t2 which are not
  pseudo-rotations.
\newblock {\em to appear in Ann. Sci. E.N.S., preprint arXiv:1611.03784}, 2016.

\bibitem[KPS16]{KoropeckiPasseggiSambarino2016FMC}
A.~Koropecki, A.~Passeggi, and M.~Sambarino.
\newblock The franks-misiurewicz conjecture for extensions of irrational
  rotations.
\newblock {\em Preprint {\tt arXiv:1611.05498}}, 2016.

\bibitem[KR17]{kocsard2017rotational}
A.~Kocsard and F.~P. Rodrigues.
\newblock Rotational deviations and invariant pseudo-foliations for periodic
  point free torus homeomorphisms.
\newblock {\em Mathematische Zeitschrift}, pages 1--25, 2017.

\bibitem[KT14a]{KoropeckiTal2012Irrotational}
A.~Koropecki and F.~Tal.
\newblock Area-preserving irrotational diffeomorphisms of the torus with
  sublinear diffusion.
\newblock {\em Proc. Amer. Math. Soc.}, 142(10):3483--3490, 2014.

\bibitem[KT14b]{KoropeckiTal2012StrictlyToral}
A.~Koropecki and F.~Tal.
\newblock Strictly toral dynamics.
\newblock {\em Invent. Math.}, 196(2):339--381, 2014.

\bibitem[KT18]{koropecki2018fully}
A.~Koropecki and F.~A. Tal.
\newblock Fully essential dynamics for area-preserving surface homeomorphisms.
\newblock {\em Ergodic Theory and Dynamical Systems}, 38(5):1791--1836, 2018.

\bibitem[LM91]{llibre/mackay:1991}
J.~Llibre and R.S. MacKay.
\newblock Rotation vectors and entropy for homeomorphisms of the torus isotopic
  to the identity.
\newblock {\em Ergodic Theory Dynam.\ Systems}, 11:115--128, 1991.

\bibitem[MZ89]{MisiurewiczZiemian1989RotationSets}
M.~Misiurewicz and K.~Ziemian.
\newblock Rotation sets for maps of tori.
\newblock {\em J. Lond. Math. Soc.}, 2(3):490--506, 1989.

\bibitem[PS19]{passeggi2018deviations}
A.~Passeggi and M.~Sambarino.
\newblock Deviations in the franks–misiurewicz conjecture.
\newblock {\em Ergodic Theory and Dynamical Systems}, page 1–8, 2019.

\end{thebibliography}

\end{document}